\documentclass[final]{svjour3}
\smartqed  
\usepackage[latin1]{inputenc}
\usepackage[english]{babel}
\usepackage{amsmath}
\usepackage{amsfonts}
\usepackage{amssymb}
\usepackage{epsfig}
\usepackage{amsopn}
\usepackage{color}
\usepackage{graphicx}
\usepackage{enumerate}
\newcommand{\real}{\mathbb{R}}

\def\a{\alpha}

\def\e{\varepsilon}        

\newcommand{\dx}{\,{\rm d}x}

\newcommand{\RR}{\mathbb{R}}

\def\qed{\,\unskip\kern 6pt \penalty 500
\raise -2pt\hbox{\vrule \vbox to8pt{\hrule width 6pt
\vfill\hrule}\vrule}\par}
\definecolor{darkblue}{rgb}{0.05, .05, .65}
\definecolor{darkgreen}{rgb}{0.1, .65, .1}
\definecolor{darkred}{rgb}{0.8,0,0}
\newcommand{\beqn}{\begin{equation}}
\newcommand{\eeqn}{\end{equation}}
\newcommand{\bear}{\begin{eqnarray}}
\newcommand{\eear}{\end{eqnarray}}
\newcommand{\bean}{\begin{eqnarray*}}
\newcommand{\eean}{\end{eqnarray*}}

\usepackage{graphicx}
\begin{document}

\title{Rotationally symmetric $p$-harmonic maps from $D^2$ to $S^2$\thanks{ The authors have been partially supported by the Spanish MEC and
FEDER, project MTM2008-03176.}
}

\titlerunning{Rotationally symmetric $p$-harmonic maps from $D^2$ to $S^2$}        

\author{Razvan Gabriel Iagar         \and
        Salvador Moll 
}

\authorrunning{Razvan Gabriel Iagar         \and
        Salvador Moll} 

\institute{R. G. Iagar \and
           S. Moll \at
              Departamento de Análisis
Matemático, Univ. de Valencia, Dr. Moliner 50, 46100, Burjassot,
Spain.  \email{razvan.iagar@uv.es, \ j.salvador.moll@uv.es} \and R.
G. Iagar \at Institute of Mathematics of the Romanian Academy, P.O.
Box 1-764, RO-014700, Bucharest, Romania.
                 }

\date{}

\maketitle

\begin{abstract}
We consider rotationally symmetric $p$-harmonic maps from the unit
disk $D^2\subset\real^2$ to the unit sphere $S^2\subset\real^3$,
subject to Dirichlet boundary conditions and with $1<p<\infty$. We
show that the associated energy functional admits a unique minimizer
which is of class $C^{\infty}$ in the interior and $C^1$ up to the
boundary. We also show that there exist infinitely many global
solutions to the associated Euler-Lagrange equation and we
completely characterize them. 
\end{abstract}
\tableofcontents
\section{Introduction and main results}

Given a domain $\Omega\subset\real^N$, a real number $p>1$, a
smoothly embedded compact submanifold without boundary $M$ of
$\real^{N+1}$, and a mapping $u:\Omega\mapsto M$, we consider the
energy functional
\begin{equation}\label{pEnergy}
E_{p}(u):=\frac{1}{p}\int_{\Omega}|\nabla u|^p\,\dx.
\end{equation}
The $p$-harmonic flow associated to $E_p$ is given by
\begin{equation}\label{pHarm}
u_t=-\pi_{u}(-\Delta_{p}u), \quad \Delta_{p}u=\hbox{div}(|\nabla
u|^{p-2}\nabla u),
\end{equation}
where, for a generic point $v\in M$, we denote by $\pi_{v}$ the
orthogonal projection of $\real^{N+1}$ onto the tangent space
$T_{v}M$ to $M$ at $v$. Notice that, because of this projection, any
solution to Eq. \eqref{pHarm} is constrained to remain in $M$. We
next particularize $M$ to be the unit sphere $S^N$ of $\real^{N+1}$. Then, \eqref{pHarm} can be written in an explicit form
\begin{equation}\label{pHarmSphere}
u_t=\hbox{div}(|\nabla u|^{p-2}\nabla u)+u|\nabla u|^p.
\end{equation}
This is a system of partial differential equations that, besides its pure
mathematical interest as involving a competition between
quasi-linear diffusion and reaction gradient terms, has been
proposed in various contexts, such as ferromagnetism \cite{dSPG},
theory of liquid crystals \cite{vdH}, multigrain problems \cite{KWC}
and image processing \cite{SapiroBook}. In this context it has been used as a prototype of quite complicated problems
containing nonlinear reaction-diffusion systems of partial differential
equations for the evolution of director fields.

From the mathematical point of view, the Dirichlet problem for Eq.
\eqref{pHarmSphere} with boundary condition $u(t,x)=u_0(x)$ for
$(t,x)\in\partial Q$, $Q:=(0,\infty)\times\Omega$, has been widely
considered in the last decades. Referenced discussions can be found in
\cite{BPP,BPvdH} for $p=2$,
\cite{chh,Hbook,Misawa} for $p\neq 2$ and
\cite{GK,dPGM,GM} for the special limit case $p=1$.
About the local and global existence of solutions to the
Dirichlet problem for Eq. \eqref{pHarmSphere}, it has been shown in
\cite{CDY} for $p=2$ and, more recently, for $p=1$ in \cite{GK,GM}
that for suitable boundary data, classical solutions exist for short
time and their first derivative at the origin blows up in finite time. In \cite{GM} a
sharp condition on the boundary data for this phenomenon to happen is given.
In all these papers, the case of rotationally symmetric and
stationary boundary conditions is considered.

We devote the present paper to the study of the steady-states
corresponding to the Dirichlet problem for Equation
\eqref{pHarmSphere} when $1<p<\infty$, $M=S^2$, the unit sphere of
$\real^3$, and $\Omega=D^2=\{(x_1,x_2)\in\real^2:
x_1^2+x_2^2\leq1\}$, the unit disk of $\real^2$, and we restrict
ourselves to deal with rotationally symmetric solutions. A
rotationally symmetric solution of \eqref{pHarmSphere} has in this
case the general form
\begin{equation}\label{rotsim}
u(t,x)=\left(\frac{x_1}{r}\sin h(t,r),\frac{x_2}{r}\sin h(t,r),\cos h(t,r)\right),
 r=|x|,  x=(x_1,x_2)\in D^2
\end{equation}
and the Dirichlet boundary condition takes the form
$h(t,1)=h(0,1)=l$. The energy functional $E_{p}$ introduced in
\eqref{pEnergy} becomes
\begin{equation}\label{pEnergyRotSim}
E_{p}(u)=\int_0^1 L_p(r,h,h_r)\,dr,
\end{equation}with a non-coercive Lagrangian given by $$L_p(r,s,\xi)=\left(r^{2/p}h_r^2 + r^{(2-2p)/p}\sin^2\,h\right)^{p/2}.$$

We replace $u$ by its special form \eqref{rotsim} in Eq.
\eqref{pHarmSphere}. By a long but straightforward calculation,
using the formulas given in the proof of Lemma 2.1 in \cite{GK}, we
obtain
\begin{equation*}
\begin{split}
\Delta_{p}u & =  (p-2)|\nabla
u|^{p-4}h_r\left[h_{r}h_{rr}-\frac{\sin^2h}{r^3}+\frac{h_r\sin h\cos
h}{r^2}\right](e^{i\theta}\cos h,-\sin h)\\& +|\nabla
u|^{p-2}\left[e^{i\theta}\left[-\left(\frac{1}{r^2}+h_r^2\right)\sin
h+\left(\frac{h_r}{r}+h_{rr}\right)\cos
h\right],\right.\\&\left.-\left(\frac{h_r}{r}+h_{rr}\right)\sin h-h_r^2\cos
h\right].
\end{split}
\end{equation*}
where we let $x=(x_1,x_2)=re^{i\theta}$. Thus, Eq.
\eqref{pHarmSphere} becomes
\begin{equation}\label{pHarmRotSim}
\begin{split}
h_t&=\left(h_r^2+\frac{sin^2\,h}{r^2}\right)^{(p-4)/2}\Big[(p-1)h_r^2h_{rr}\\ &+(p-3)\left[\frac{h_r^2\sin\,h\cos\,h}{r^2}-\frac{h_r\sin^2\,h}{r^3}\right]\\
&+\left.\frac{h_r^3}{r}+\frac{h_{rr}\sin^2\,h}{r^2}-\frac{\sin^3\,h\cos\,h}{r^4}\right].
\end{split}
\end{equation}
Notice that, in the special cases $p=1$ and $p=2$, the previous
equation simplifies to the equations in \cite[p. 535]{dPGM},
respectively \cite[p. 96]{BPvdH}.

In the present paper we investigate the rotationally symmetric
minimizers of $E_p$ subject to $h(1)=l$, as well as its smooth
critical points, that is, the steady states of Eq.
\eqref{pHarmRotSim}. We thus fill in the gap between the limit cases
$p=1$ (studied in \cite{dPGM}) and $p=2$ (studied in \cite{BPvdH}),
and moreover we extend our study to $2<p<\infty$, giving at the same
time an idea of how these results change with respect to $p$. By the
periodicity of $\sin^2\,h$ and the symmetry with respect to $\pi/2$,
we may assume without loss of generality that $l\in(0,\pi/2]$. We
incorporate the boundary condition by defining
\begin{equation}\label{energy functional}
G_{p,l}(h):=\left\{\begin{array}{ll}E_{p}(h), \quad \hbox{if} \
h(1)=l,\\ \\ +\infty \quad \hbox{otherwise}.\end{array}\right.
\end{equation}
We prove first existence and uniqueness of the minimizer of the energy
functional.
\begin{theorem}\label{th.1}
Let $p>1$ and $l\in(0,\pi/2]$. Then, up to a reflection with
respect to $h=\pi/2$ in the case $l=\pi/2$, there exists a unique
solution $h_{l}$ to \begin{equation}
  \label{P} h_{l}\in {\rm \ argmin }\{G_{p,l}(h) : h \in W^{1,1}(0,1)\cap W^{1,p}_{loc}(0,1)\}\,.
\end{equation}
 Furthermore, $h_{l}\in
C^{\infty}((0,1])\cap C^1([0,1])$, $h_{l}$ is positive and
increasing in $(0,1]$, $h_{l}(0)=0$, $h_{l}(1)=l$ and $h_{l}$ solves
the following ordinary differential equation
\begin{equation}\label{ODE}\begin{split}
(p-1)h'^2h''+(p-3)\left(h'^2\frac{\sin\,h\cos\,h}{r^2}-h'\frac{\sin^2\,h}{r^3}\right)\\ +\frac{h'^3}{r}+\frac{h''\sin^2\,h}{r^2}-\frac{\sin^3\,h\cos\,h}{r^4}=0,
\end{split}
\end{equation}
in a classical sense.
\end{theorem}

We also show that the functional $G_p$ admits an infinite
sequence of smooth critical points which are different from the
minimizer. We distinguish between the cases $1<p<2$ and $p>2$.

\begin{theorem}\label{th.2}
Given $p\in(1,2)$, there exists a global solution $h\in
C^{\infty}((0,\infty))\cap C^1([0,\infty))$ to Eq. \eqref{ODE} in a
classical sense, which satisfies the following properties:

\noindent (a) $h(0)=0$, $\lim\limits_{r\to\infty}h(r)=\pi/2$;

\noindent (b) There exists an increasing sequence of critical points
$r_n$ of $h$ such that $r_n\to\infty$ as $n\to\infty$,
$h(r_{2n})\in(\pi/2,\pi)$ and are local maxima,
$h(r_{2n+1})\in(0,\pi/2)$ and are local minima. Moreover, there
exists some positive integer $n_0$ sufficiently large such that
$|h(r_n)-\pi/2|$ is decreasing for $n\geq n_0$, and
\begin{equation}\label{MaxOsc}
\left|h(r_0)-\frac{\pi}{2}\right|=\max\left\{\left|h(r_n)-\frac{\pi}{2}\right|:n\geq 0\right\};
\end{equation}

\noindent (c) For any $l\in(0,\pi/2)$, the minimizer $h_l$ of
$G_{p}$ satisfies $h_{l}(r)=h(\a r)$ (or $h_{l}(r)=\pi-h(\a r)$ if
$l=\pi/2$), for some suitable $\a>0$;

\noindent (d) Any non-constant finite global solution $\tilde{h}$ to
Eq. \eqref{ODE} has the form $\tilde{h}(r)=k\pi\pm h(\a r)$, for
some $\a>0$ and $k\in \mathbb Z$.

\noindent (e) There are infinitely many global solutions which have
$-\infty$ (respectively $+\infty$) as limit when $r\to 0^+$. They
are either increasing (resp. decreasing) forever and with limit either $+\infty$ (resp. $-\infty$) or
$k\pi$ for some $k\in \mathbb Z$ as $r\to +\infty$ or they increase
(resp. decrease) up to some $r_*\in (0,+\infty)$ and after $r_*$
they oscillate around $(2k+1)\frac{\pi}{2}$ for some $k\in\mathbb Z$
as in (b).
\end{theorem}
In order to illustrate graphically these solutions and ease the
understanding of the results, we plot them in the following figures.

\medskip

\begin{figure}\label{fig.1}
 $$\includegraphics[width=0.47\textwidth]{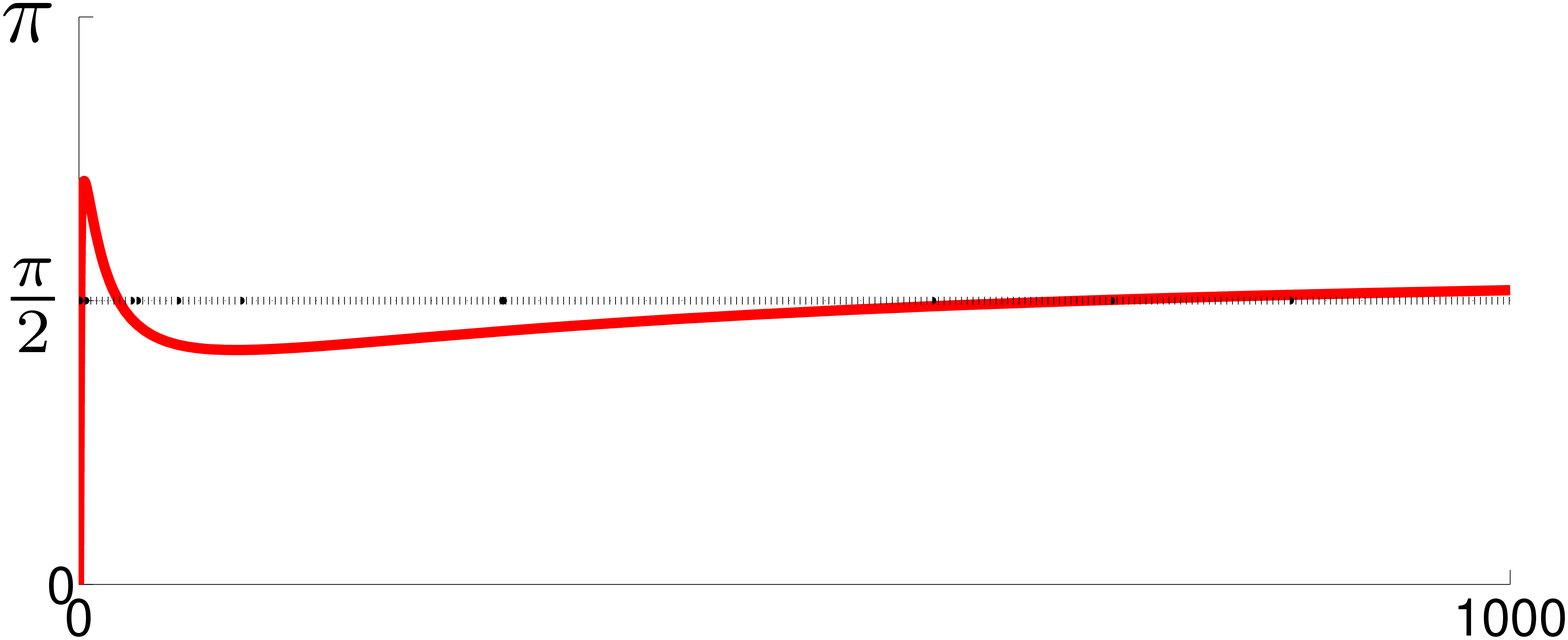}\,\, \includegraphics[width=0.47\textwidth]{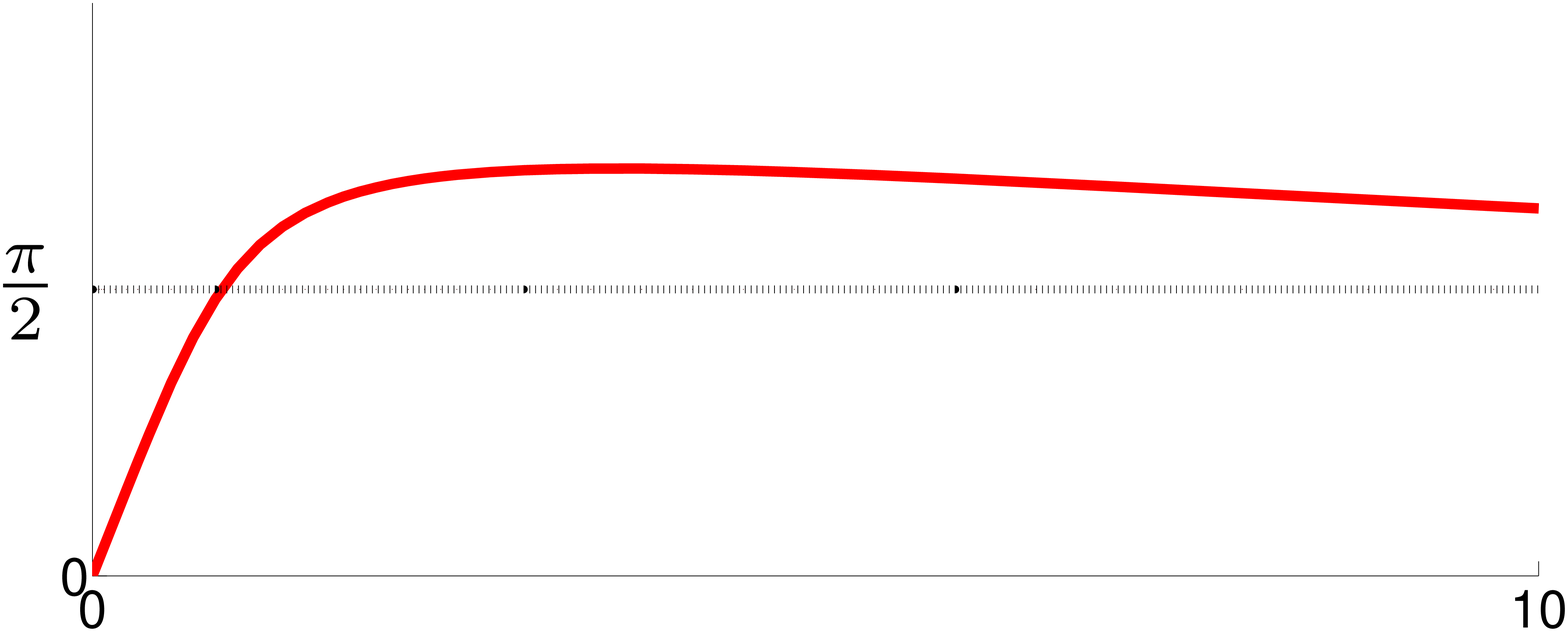}$$
   \caption{Global solution to \eqref{ODE} for $1<p<2$ as given by Theorem \ref{th.2} (a) and (b). Numerical simulation for $p=1+\frac{\sqrt 3}{3}$. Left: Global solution. Right: zoom of the global solution at the origin.}
\end{figure}
\begin{figure}\label{fig2}
$$ \includegraphics[width=0.47\textwidth]{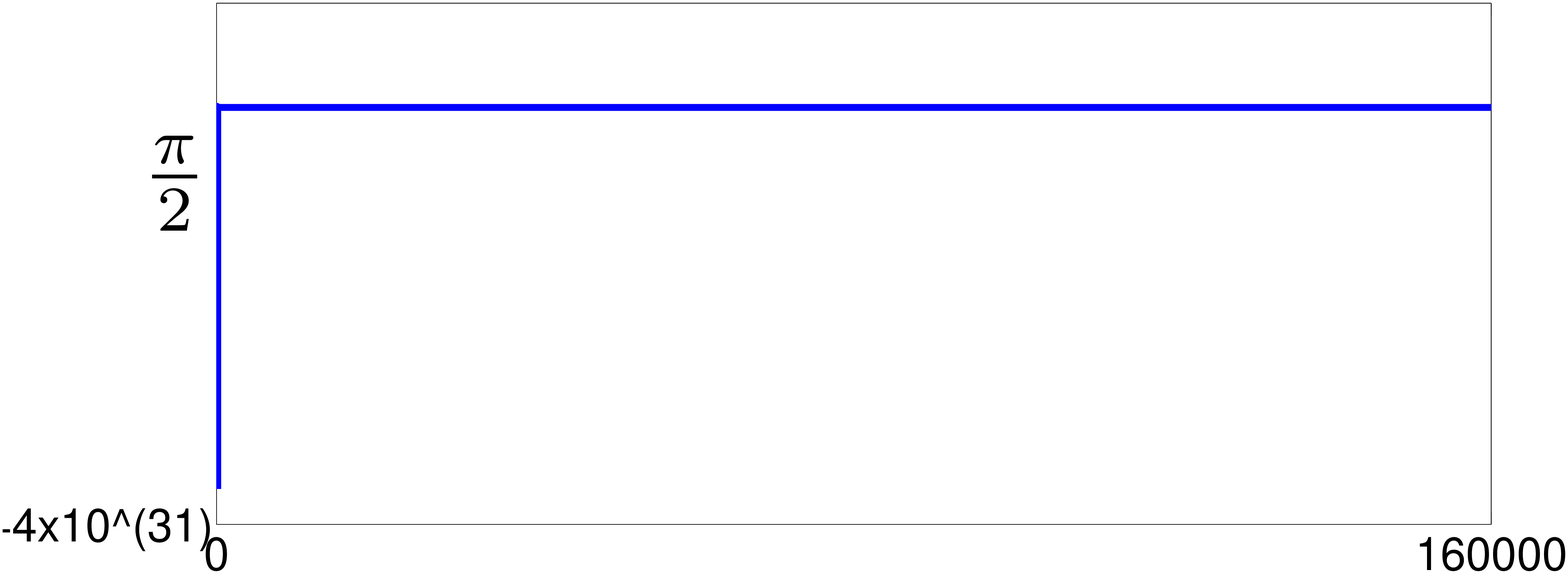}\,\, \includegraphics[width=0.47\textwidth]{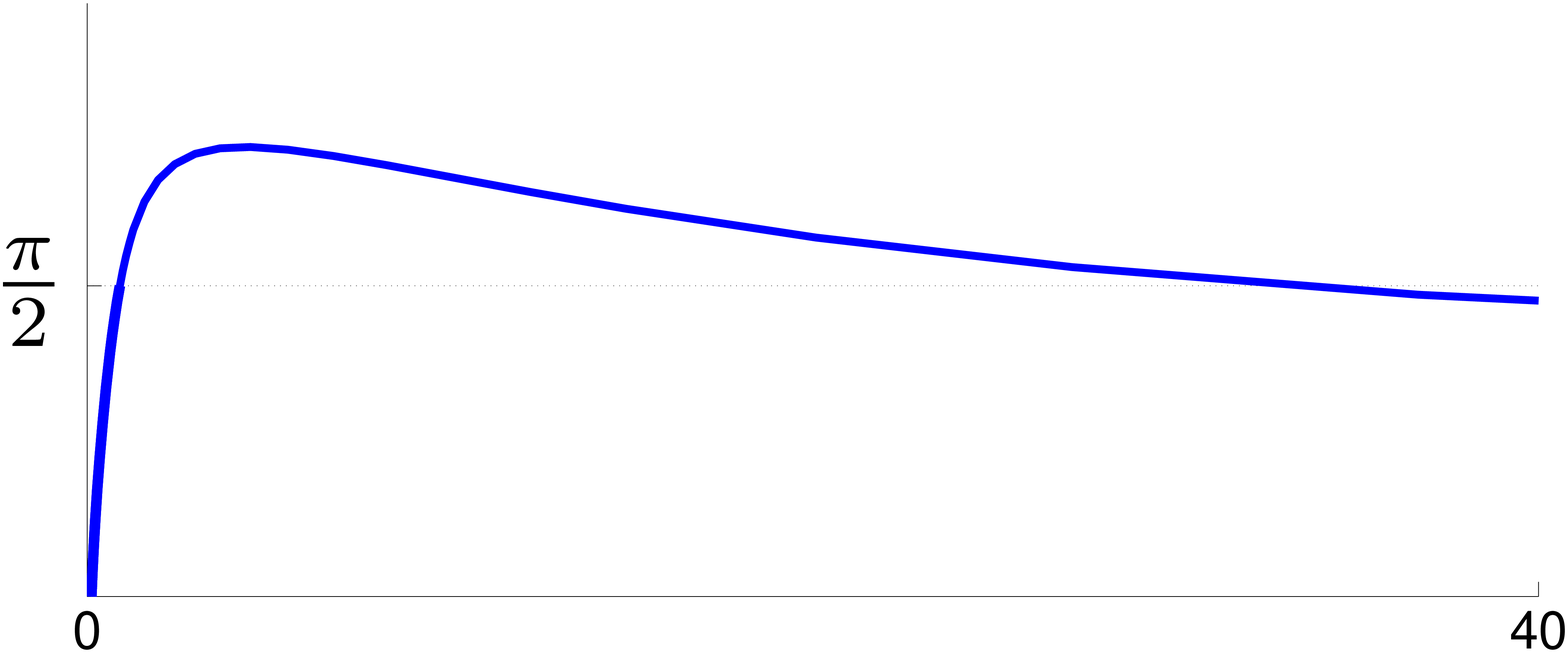}
  $$ \caption{Global solution to \eqref{ODE} for $1<p<2$ as given by Theorem \ref{th.2} (e). Numerical simulation for $p=1+\frac{\sqrt 3}{3}$. Left: Global solution. Right: zoom of the global solution at the origin.}
\end{figure}

Due to the scaling invariance of Eq. \eqref{ODE} with respect to
$r$, Theorem \ref{th.2} shows that, for $1<p<2$, in addition to the minimizer
given in Theorem \ref{th.1} and to the constant solution
$h\equiv\pi/2$, there exists an infinite sequence of critical points
for $G_p$ subject to $h(1)=\pi/2$. This is similar to the limit case
$p=1$, as proved in \cite{dPGM}, but in strong contrast with harmonic maps; i.e. $p=2$, where the minimizer is the only global monotone solution connecting 0 to $\pi/2$.

For the case $p>2$, the picture is completely different and it is surprisingly very different from the case $p=2$.

\begin{theorem}\label{th.3}
Let $p>2$. Then there exists a global solution $h\in
C^{\infty}((0,\infty))\cap C^1([0,\infty))$ to Eq. \eqref{ODE} in a
classical sense, which satisfies the following properties:

\noindent (a) $h(0)=0$, $h$ is increasing and $\lim\limits_{r\to\infty}h(r)=+\infty$;

\noindent (b) For any $l\in(0,\pi/2)$, the minimizer $h_l$ of
$G_{p}$ satisfies $h_{l}(r)=h(\a r)$ (or $h_{l}(r)=\pi-h(\a r)$ if
$l=\pi/2$), for some suitable $\a>0$;

\noindent (c) Any non-constant global solution $\tilde{h}$ to Eq.
\eqref{ODE} such that $\tilde h(0)=k\pi$ for some $k\in \mathbb Z$,
has the form $\tilde{h}(r)=k\pi\pm h(\a r)$, for some $\a>0$.

\noindent (d) There are infinitely many global solutions to
\eqref{ODE}. They verify $h(0)=\frac{(2k+1)\pi}{2}$ for some
$k\in\mathbb Z$. They may oscillate around $\frac{(2k+1)\pi}{2}$ up
to some $r_*\in (0,+\infty)$ and then either they increase up to
$2(k+1)\pi$ or to $+\infty$ or they decrease up to $2k\pi$ or to
$-\infty$.

\noindent (e) There are infinitely many global solutions which have
$-\infty$ (respectively $+\infty$) as limit when $r\to 0^+$. They
are either increasing (resp. decreasing) forever and with limit
$+\infty$ (resp. $-\infty$).
\end{theorem}
We plot the most interesting solutions given by Theorem \ref{th.3}
in the following figures.

\medskip

\begin{figure}[htb]
$$ \includegraphics[width=0.47\textwidth]{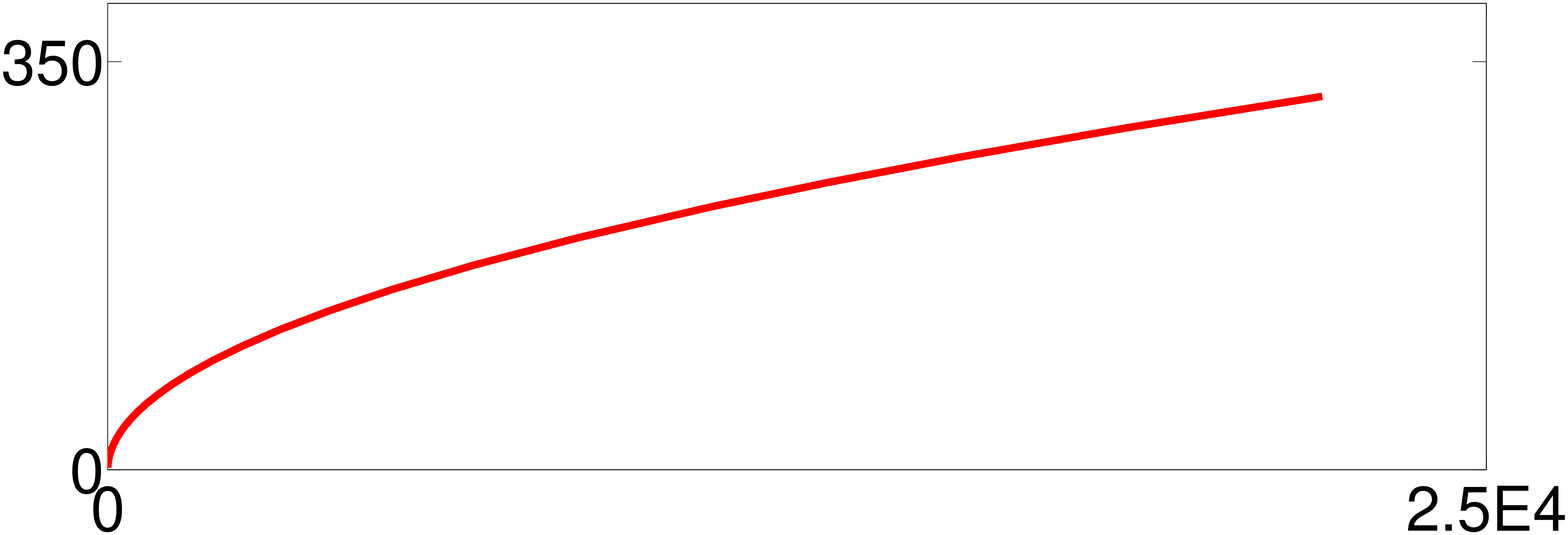}\,\, \includegraphics[width=0.47\textwidth]{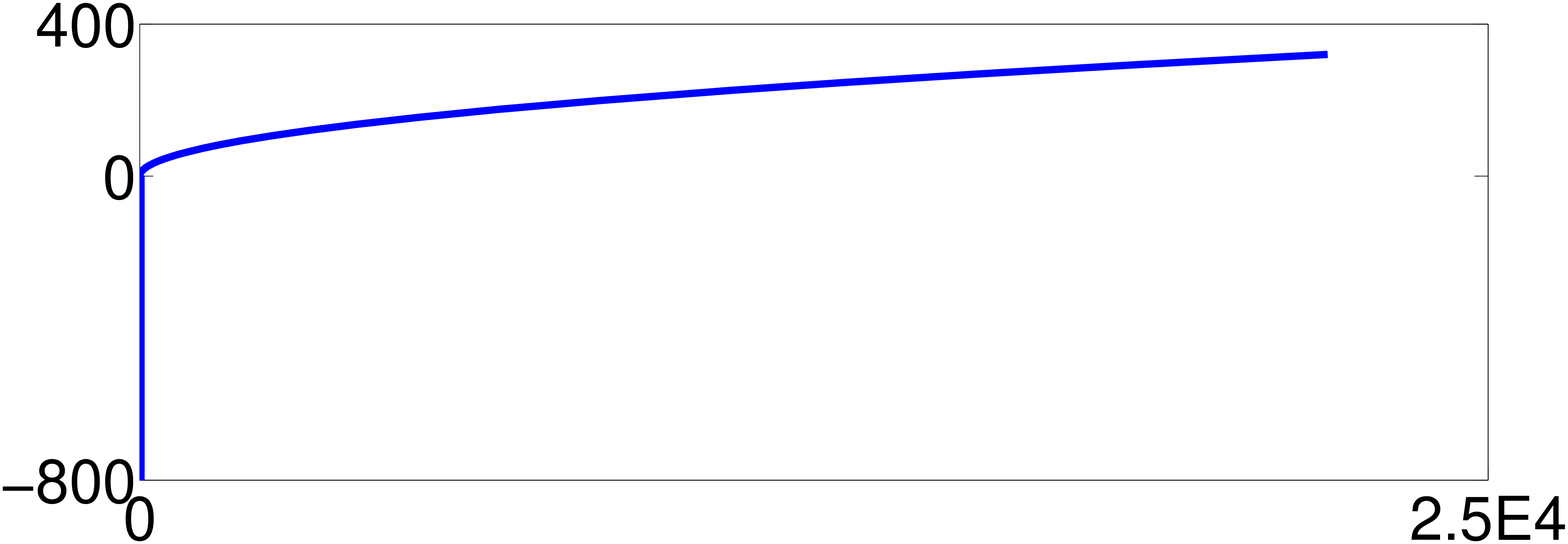}
  $$ \caption{Global solution to \eqref{ODE} for $p>2$ as given by Theorem \ref{th.3}. Numerical simulation for $p=3$. Left: global solution corresponding to case (a). Right: global solution corresponding to case (e).}
\end{figure}
\begin{figure}[htb]
$$ \includegraphics[width=0.47\textwidth]{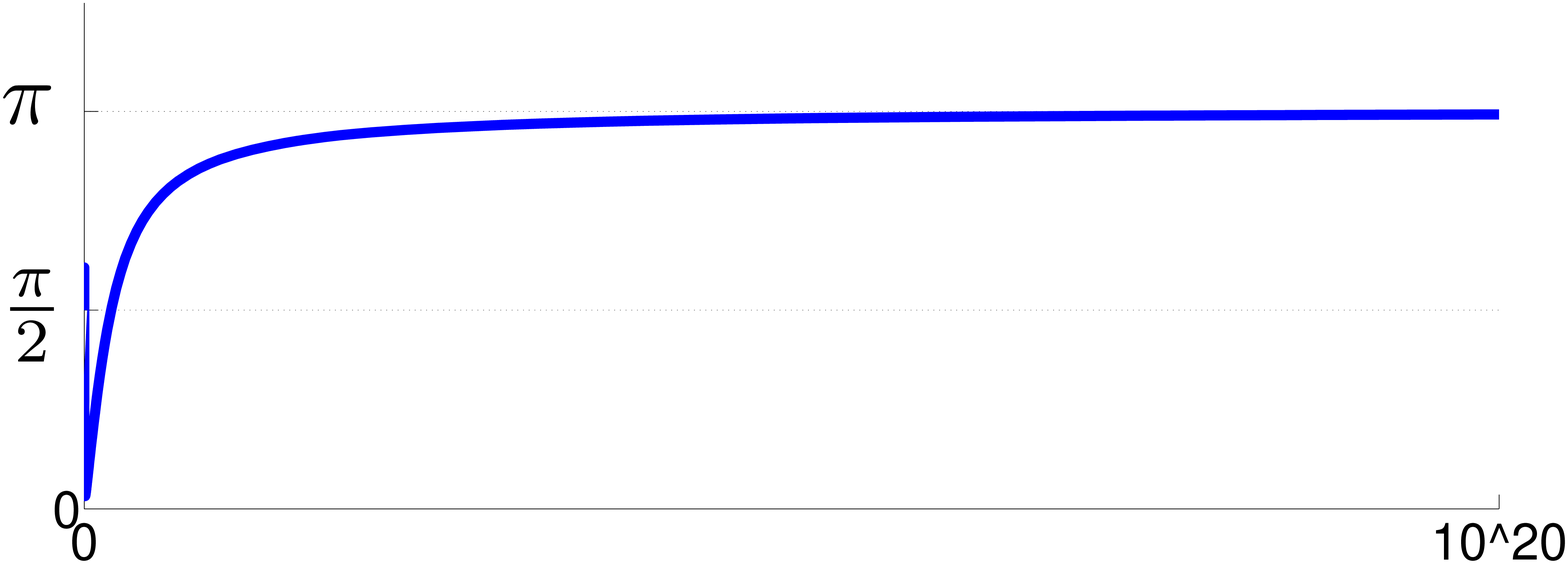}\,\, \includegraphics[width=0.47\textwidth]{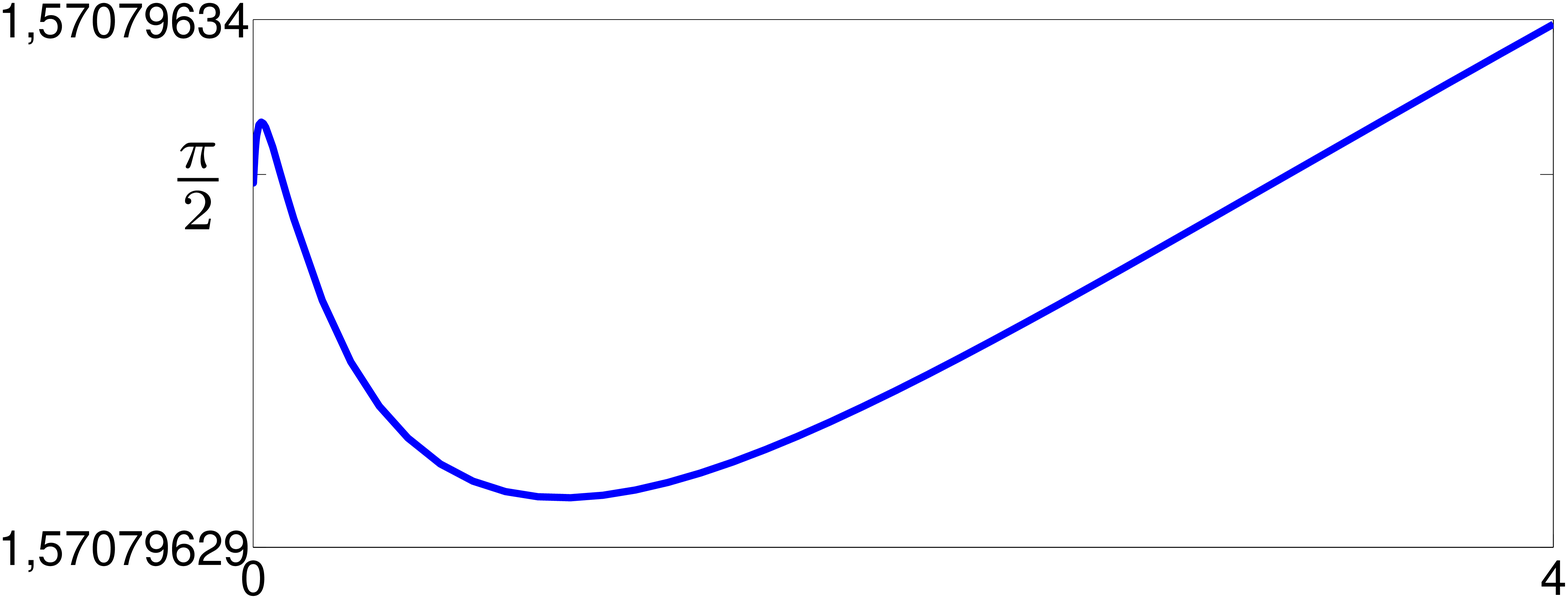}
   $$ \caption{Global solution to \eqref{ODE} for $p>2$ as given by Theorem \ref{th.3} (d). Numerical simulation for $p=2.1$. Left: global solution. Right: zoom of the global solution at the origin.}
\end{figure}

The paper is organized as follows: first, in Section \ref{sec.var}
we show the existence (and regularity) of a minimizer through
variational techniques. Then, we restrict to solutions of Eq.
\eqref{ODE} and classify them, in order to prove our main results.
This is done by the analysis of a rather complicated phase plane
associated to \eqref{ODE} after various changes of variable. In a
first step, we study in Section \ref{sec.lb0} the behavior of
solutions $h$ of \eqref{ODE} as $r\to0$, which allows us to complete
then the proof of Theorem \ref{th.1} by showing the uniqueness of
monotone solutions with the expected behavior for \eqref{ODE}; the
latter is done in Section \ref{sec.unique}. In a second and final
step, in Section \ref{sec.cp}, we prove Theorems \ref{th.2} and
\ref{th.3} by analyzing the behavior of solutions of Eq. \eqref{ODE}
for large $r$.

\section{Existence of a smooth minimizer}\label{sec.var}

First of all, we collect some properties of solutions to problem
(\ref{P}). We state the next result without proof since its proof is
identical to that of \cite[Lemmas
1--3]{dPGM}.\begin{lemma}\label{lempro} Up to a reflection with
respect to $h_l=\frac{\pi}{2}$ if $l=\frac{\pi}{2}$, any solution
$h_l$ to problem (\ref{P}) is  nondecreasing and $h_l([0,1])\subset
[0,l]$.
\end{lemma}

Next, we show that any solution of \eqref{P} must be positive in $(0,1]$.

\begin{lemma}
  \label{positivity} Let $h_l\in W^{1,1}(0,1)\cap W^{1,p}_{loc}((0,1])$ with $h_l(1)=l$ be a solution to \eqref{P}. Then, $h_l$ is positive in $(0,1]$.
\end{lemma}

\begin{proof}
By Lemma \ref{lempro}, we assume without loss of
generality that $h_l\in [0,l]$ is nondecreasing. Assume by contradiction that $0<a<1$ exists such that
$h_l(a)=0$ and $h_l(r)>0$ for $a<r\leq 1$.
Let then $\varepsilon>0$ such that $a-\varepsilon\geq 0$ and $a+\varepsilon\leq 1$. We
first notice that
\begin{equation}\label{estmenergy1}
\int_0^{a+\varepsilon}r\left[((h_l)_r)^2+\frac{\sin^2
h_l}{r^2}\right]^\frac{p}{2}dr>
\int_a^{a+\varepsilon}r((h_l)_r)^p\,dr \geq \min_{ h\in \mathcal
W}\int_{a}^{a+\varepsilon}r( h_r)^p\,dr, \end{equation}with
$$\mathcal W:= \left\{h\in W^{1,p}(a,a+\varepsilon) :   h
{\rm \ nondecreasing },   h(a)=0,
h(a+\varepsilon)=h_l(a+\varepsilon)\right\}.$$We next estimate the
last integral in \eqref{estmenergy1}. In this case, the
corresponding Lagrangian is coercive and convex. Therefore, (see e.g
\cite[Theorem 3.7]{ButGia}) there exists a minimizer $h_a\in
W^{1,p}(a,a+\varepsilon)$. Moreover, it attains the boundary values
and it verifies the corresponding Euler-Lagrange equation:
  $$r((h_a)_r)^{p-1}=C\iff h_a(r)=\frac{(1-p)C^{\frac{1}{p-1}}r^{\frac{2-p}{1-p}}}{2-p}+D.$$Substituting with the boundary values we get:
  $$C=\left[\frac{(p-2)(h_l(a+\varepsilon))[(a+\varepsilon)^\frac{p-2}{p-1}-a^\frac{p-2}{p-1}]^{-1}}{p-1}\right]^{p-1}.$$
  Then, $$\min_{h\in \mathcal W}\int_{a}^{a+\varepsilon}r( h_r)^p\,dr=\int_{a}^{a+\varepsilon} C^\frac{p}{p-1}r^\frac{-1}{p-1}\,dr=\left[\frac{(p-1)C^\frac{p}{p-1}r^\frac{p-2}{p-1}}{p-2}\right]_{a}^{a+\varepsilon}$$
  $$=\frac{(h_l(a+\varepsilon))^p(p-2)^{p-1}}{(p-1)^{p-1}}((a+\varepsilon)^\frac{p-2}{p-1}-a^\frac{p-2}{p-1})^{-p}((a+\varepsilon)^\frac{p-2}{p-1}-a^\frac{p-2}{p-1})$$
  $$=h_l(a+\varepsilon)^p(a+\varepsilon)^{2-p}\left(\frac{p-2}{p-1}\right)^{p-1}\left[1-\left(\frac{a+\varepsilon}{a}\right)^\frac{2-p}{p-1}\right]^{1-p}.$$

On the other hand, if we take
$$
\theta(r)=2\arctan\left(\frac{\tan(\frac{h_l(a+\varepsilon)}{2})r}{(a+\varepsilon)}\right),
$$
we have
$$
\int_0^{a+\varepsilon}\sqrt{r^2(\theta')^2+\sin^2 \theta}\,dr$$$$=
\frac{2^\frac{p}{2}(a+\varepsilon)^{2-p}\tan\left(\frac{h_l(a+\varepsilon)}{2}\right)^p\left[\cos^{2p}\left(\frac{h_l(a+\varepsilon)}{2}\right)-
\cos^2\left(\frac{h_l(a+\varepsilon)}{2}\right)\right]}{(1-p)\sin^2(\frac{h_l(a+\varepsilon)}{2})}.
$$
Therefore, letting
$$
h^*(r)=\left\{\begin{array}{ll} \theta(r), & 0<r<a+\varepsilon,
\\ \\
h_l(r), & \mbox{elsewhere},
\end{array}
\right.
$$
we obtain that
$$
E_{p,l}(h^*) - E_{p,l}(h_l)<
h_l^p(a+\varepsilon)(a+\varepsilon)^{2-p}\times$$$$\left[-\left[\frac{(2-p)a^\frac{2-p}{p-1}}
{(p-1)((a+\varepsilon)^\frac{2-p}{p-1}-a^\frac{2-p}{p-1})}\right]^{p-1}+2^{-\frac{p}{2}}+o_\varepsilon(1)\right],
$$
as $\e\to0$. Finally note that
$$\left[\frac{(2-p)a^\frac{2-p}{p-1}}
{(p-1)((a+\varepsilon)^\frac{2-p}{p-1}-a^\frac{2-p}{p-1})}\right]^{p-1}\stackrel{\varepsilon\to
0}\to +\infty$$ for $p\neq 2$. Hence $h_l$ cannot be a
minimizer.
\end{proof}

We next show that any solution to problem \eqref{P} is smooth in the bulk.
\begin{lemma}\label{lem-regularity}
  Let $h_l\in W^{1,1}(0,1)\cap W^{1,p}_{loc}((0,1])$ with $h_l(1)=l$ be a solution to \eqref{P}. Then, $h_l\in C^\infty((0,1])$.
\end{lemma}
\begin{proof}
  In order to improve the regularity of $h_l$, we claim that for any $\varepsilon>0$,
  $h_l$ is a minimizer of \begin{equation}\label{localmin}\int_\varepsilon^1 L_p(r,h,h_r)\end{equation} within the space of functions $h$ such that $h\in W^{1,p}(\varepsilon,1)\,, h(\varepsilon)=h_l(\varepsilon), h(1)=l, h {\rm \ nondecreasing}$. Let $$W^{1,p}_{\varepsilon,h_l}:=\left\{ h\in W^{1,p}(\varepsilon,1)\,, h(\varepsilon)=h_l(\varepsilon), h(1)=l\right\}.$$
  Then, since the Lagrangian in \eqref{localmin} is superlinear, coercive, analytic on $(\varepsilon,1)\times [h_l(\varepsilon),+\infty)\times\RR$, convex with respect to $h_r$ and with bounded derivatives with respect to $h$ and $h_r$ (by Lemmas \ref{lempro} and \ref{positivity}),  then (\cite[Theorems 3.7 and 4.1]{ButGia}) there exists $h_\varepsilon\in W^{1,p}_{\varepsilon,h_l}$ such that $$h_\varepsilon\in {\rm argmin \ }\left\{\int_\varepsilon^1 L_p(r,h,h_r) : h\in W^{1,p}_{\varepsilon,h_l}(\varepsilon,1)\,, h {\rm \ nondecreasing}\right\}.$$ Moreover, any minimizer is of class $C^\infty([\varepsilon,1])$. Observe that, once the claim is proved, then we obtain that $$h_l\in W^{1,1}(0,1)\cap C^\infty((0,1]).$$

  Suppose by contradiction that the claim does not hold. Then, defining $$h^{*}_\varepsilon(r):=\left\{\begin{array}
    {cc} h_l(r), \quad & {\rm if \ } 0\leq r\leq \varepsilon, \\ \\ h_\varepsilon(r), \quad & {\rm if \ } \varepsilon<r\leq
    1,
  \end{array}\right.$$ we would obtain that $h^{*}_\varepsilon\in W^{1,1}(0,1)\cap W^{1,p}_{loc}((0,1])$ and $$E_p(h^{*}_\varepsilon)=\int_0^\varepsilon L_p(r,h^*,(h_r^*))\,dr+\int_\varepsilon^1 L_p(r,h_\varepsilon,(h_\varepsilon)_r)\,dr< E_p(h_l)\,,$$ which contradicts $h_l$ verifying \eqref{P}.
\end{proof}

These previous Lemmas permit us to prove the following existence result.
\begin{proposition}
  \label{existence} There exists a solution $h$ of \eqref{P} such that $h\in W^{1,1}(0,1)\cap C^\infty((0,1])$.
\end{proposition}
\begin{proof}
  Since the Lagrangian is not coercive in $[0,1]$, the existence of the minimizer does not follow from classical results in the literature. We follow the direct method in the calculus of variations.

  Let $\{h_n\}\subset W^{1,p}(0,1)$ be a minimizing sequence such that $h_n(1)=l$. By Lemma \ref{lempro}, we may assume that each $h_n$ is nondecreasing and $h_n([0,1])\subset [0,l]$. Then, $$\|h_n\|_{W^{1,1}(0,1)}=\|h_n\|_1+\|(h_n)_r\|_1\leq 2l.$$
  Thus, we can extract a subsequence, not relabeled, such that $h_n\to h^*\in BV(0,1)$ in $L^1(0,1)$. We also obtain that $h_n\to h$ a.e. in $[0,1]$.
  Moreover, given an $\varepsilon>0$, since $$\int_\varepsilon^1 (h_n)_r^p\,dr\leq \frac{1}{\varepsilon}\int_\varepsilon^1 L(r,h_n,(h_n)_r)\,dr\leq \frac{E_p(h_n)}{\varepsilon}<+\infty,$$then there is a subsequence of the previous subsequence such that $h_{n_\varepsilon}\to h^{**}$ in $W^{1,p}(\varepsilon,1)$ and $h^{**}(1)=l$.
  Therefore, since the whole sequence converges to $h^*$, we obtain that $h^*\in BV(0,1)\cap W^{1,p}_{loc}((0,1])$ and $h^*(1)=l$. Note that this implies that $h^*\in W^{1,1}(0,1)\cap  W^{1,p}_{loc}((0,1])$. Moreover, since the Lagrangian is a Carath\'eodory function, then $$\int_\varepsilon L(r,h^*,h_r^*)\,dr\leq\liminf_{n\to \infty}\int_\varepsilon^1 L(r,h_n,(h_n)_r\,dr\leq \liminf_{n\to \infty}E_p(h_n).$$Finally, by the monotone convergence theorem we conclude that $h^*$ is a minimizer since $$E_p(h^*)\leq \liminf_{n\to\infty} E_p(h_n).$$
  We finish the proof by applying Lemma \ref{lem-regularity}.
  \end{proof}

\section{Local behavior at $r=0$}\label{sec.lb0}

We begin with the systematic study of the steady states given as
solutions to \eqref{ODE}. In the present section we deal with their
behavior near the origin, in both cases $1<p<2$ and $p>2$.

\subsection{The case $1<p<2$}\label{subsec.loc0}

The main result in this Section is the following one.
\begin{proposition}\label{prop.loc0}
Let $p\in(1,2)$ and $h$ be a nonconstant solution to \eqref{ODE} in
an interval $(0,\e)$ for some $\e>0$, and such that $h(r)\in(0,\pi)$
for all $r\in(0,\e)$. Then $h\in C^1([0,\e))$, $h$ is monotone in a
right-neighborhood of $r=0$ and
\begin{equation}\label{lim0.ODE}
\lim\limits_{r\to0^+}h(r)=0 \quad {\rm or} \quad
\lim\limits_{r\to0^+}h(r)=\pi.
\end{equation}
\end{proposition}
The proof of this proposition will be the consequence of some
suitable changes of variable and a phase plane analysis of the
resulting algebraic autonomous system. The changes of variable coincide with the ones in the particular case $p=1$ \cite{dPGM}. First,
we pass to logarithmic coordinates by letting $f(t):=h(e^{-t})$. Then, \eqref{ODE} becomes
\begin{equation}\label{ODE2}\begin{split}
(p-1)f'^2(f''+f') & +(p-3)[f'^2\sin\,f\cos\,f+f'\sin^2\,f]-f'^3 \\ & +(f''+f')\sin^2\,f-\sin^3\,f\cos\,f=0.\end{split}
\end{equation}
For $f\in(0,\pi)$, as supposed in Proposition \ref{prop.loc0}, we
make the further change
$$w(t):=\cot\,f(t),$$
in order to eliminate the trigonometric terms. We obtain the
following algebraic differential equation
\begin{equation}\label{ODE3}
w''=\frac{(1+w^2+w'^2)(2(p-1)ww'+(2-p)(1+w^2))w'}{(1+w^2)(1+w^2+(p-1)w'^2)}-w,
\end{equation}
which can be written in an equivalent form as the autonomous system
\begin{equation}\label{ODEsyst1}
\left\{\begin{array}{ll}\displaystyle w'=k, \\ \\ \displaystyle
k'=\frac{(1+w^2+k^2)(2(p-1)kw+(2-p)(1+w^2))k}{(1+w^2)(1+w^2+(p-1)k^2)}-w.\end{array}\right.
\end{equation}
The system \eqref{ODEsyst1} has a unique critical point in the
plane, the origin. The linearized system around the origin has the
matrix
$$
A(0,0)=\left(
  \begin{array}{cc}
    0 \ & 1 \\
    -1 \
     & 2-p \\
  \end{array}
\right).
$$
Thus, the point is an unstable node, the matrix $A(0,0)$ having two
complex eigenvalues with real part $(2-p)/2>0$, as it is easy to
check. Then, all the orbits of \eqref{ODEsyst1} come out from the
origin and will connect to critical points at infinity. In order to
study the critical points of \eqref{ODEsyst1} at infinity, we note that \eqref{ODEsyst1} is not polynomial, thus
the standard theory cannot be used directly. Instead, we write
\begin{equation}\label{interm14}
\begin{split}
\frac{dk}{dw}&=\frac{(1+w^2+k^2)\left[2(p-1)wk+(2-p)(1+w^2)\right]k}{k(1+w^2)(1+w^2+(p-1)k^2)}\\ &-\frac{w(1+w^2)(1+w^2+(p-1)k^2)}{k(1+w^2)(1+w^2+(p-1)k^2)}=\frac{Q(w,k)}{P(w,k)},
\end{split}
\end{equation}
where $Q(w,k)$, $P(w,k)$ are both polynomials of degree 5. Thus, we
can follow the strategy indicated in \cite[pp. 270-271]{Pe}. We
let
$$
P(w,k)=P_1(w,k)+\ldots+P_5(w,k), \quad
Q(w,k)=Q_1(w,k)+\ldots+Q_5(w,k),
$$
with $P_j$, $Q_j$ homogeneous polynomials in $(w,k)$ of degree $j$,
$1\leq j\leq5$. Then, Theorem 1 in \cite[p. 271]{Pe} ensures that the
critical points at infinity occur at the polar angles $\theta$
solving the equation
\begin{equation}\label{critinf1}
\cos\theta Q_5(\cos\theta,\sin\theta)-\sin\theta
P_5(\cos\theta,\sin\theta)=0.
\end{equation}
Let us notice that, in our case,
\begin{equation*}
\begin{split}
&Q_5(w,k)=k(w^2+k^2)\left[2(p-1)wk+(2-p)w^2\right]-w^3(w^2+(p-1)k^2),\\
&P_5(w,k)=kw^2(w^2+(p-1)k^2).
\end{split}
\end{equation*}
Hence \eqref{critinf1} becomes
\begin{equation}\label{critinf2}
\begin{split}
& \cos\theta\sin\theta\left[2(p-1)\sin\theta\cos\theta+(2-p)\cos^2\theta\right]\\ &-\cos^4\theta\left[\cos^2\theta+(p-1)\sin^2\theta\right]=\cos^2\theta\sin^2\theta\left[\cos^2\theta+(p-1)\sin^2\theta\right].
\end{split}
\end{equation}
We thus get six critical points at infinity, corresponding to polar
angles
\begin{equation}\label{polar}
\theta\in\left\{\frac{\pi}{2},\frac{3\pi}{2},\frac{\pi}{4},\frac{5\pi}{4},\arccos\left(-\frac{p-1}{\sqrt{p^2-2p+2}}\right),-\arccos\left(\frac{p-1}{\sqrt{p^2-2p+2}}\right)\right\}.
\end{equation}
Since $w'=k$, we cannot have at the same time $k\to\infty$ and
$w\to-\infty$ or $w\to\infty$ and $k\to-\infty$ as $t\to\infty$,
that is, the product $\sin\theta\cos\theta$ must be nonnegative at
infinity. We thus reduce, also by symmetry, our analysis to the
study of the two points corresponding to polar angles $\theta=\pi/4$
and $\theta=\pi/2$.

We start with the local analysis of the system \eqref{ODEsyst1} near
the critical point with $\theta=\pi/2$, which is more involved. By
Theorem 2 in \cite[pp. 272-273]{Pe}, the flow in a neighborhood of
this point is topologically equivalent to the flow near the origin
$(x,z)=(0,0)$ of the following system
\begin{equation*}
\begin{split}
&\pm
x'=xz^5Q\left(\frac{x}{z},\frac{1}{z}\right)-z^5P\left(\frac{x}{z},\frac{1}{z}\right),\\
&\pm z'=z^6Q\left(\frac{x}{z},\frac{1}{z}\right),
\end{split}
\end{equation*}
the signs plus or minus being determined by the flow on the equator
of $S^{2}$ as indicated in Theorem 1, \cite[Section 3.10]{Pe}. After
straightforward calculations, the system above can be written in the
form
\begin{equation}\label{critsyst1}
\begin{split}
&\pm x'=(p-1)(x^2-z^2)+{\rm higher \ order \ terms},\\
&\pm z'=2(p-1)xz+{\rm higher \ order \ terms}.
\end{split}
\end{equation}
We will use next the theory developed in
\cite{LPR96} to show that indeed the system \eqref{critsyst1} is
topologically equivalent locally near $(x,z)=(0,0)$ to the
homogeneous system
\begin{equation}\label{critsyst2}
\begin{split}
&\pm x'=(p-1)(x^2-z^2),\\
&\pm z'=2(p-1)xz.
\end{split}
\end{equation}
In order to do this, we begin by passing in \eqref{critsyst2} to
polar coordinates $(r,\theta)$
so that, after easy calculations, the system becomes
$$
\frac{dr}{dt}=r^2f(\theta), \quad \frac{d\theta}{dt}=rg(\theta),
$$
where
\begin{equation*}
f(\theta):=(p-1)\cos\theta, \quad g(\theta):=(p-1)\sin\theta.
\end{equation*}
Introducing a new time $s$ via $ds/dt=r$ and making a further change
of variable $\varrho:=r/(1+r)$, we transform \eqref{critsyst2} into
the topologically equivalent system
\begin{equation}\label{critsyst3}
\begin{split}
&\pm \varrho'=(p-1)\varrho(1-\varrho)\cos\theta,\\
&\pm \theta'=(p-1)\sin\theta,
\end{split}
\end{equation}
where derivatives are taken with respect to $s$, and
$$(\varrho,\theta)\in
D:=\{(\varrho,\theta):0\leq\varrho<1,\theta\in[0,2\pi)\}.$$ This
system has two critical points on $\partial D$, namely
$(\varrho,\theta)\in\{(1,0),(1,\pi)\}$. The linearization of
\eqref{critsyst3} near these two critical points has the matrices
$$
M(1,0)=\left(
         \begin{array}{cc}
           -(p-1) & 0 \\
           0 & (p-1) \\
         \end{array}
       \right), \quad M(1,\pi)=\left(
                                 \begin{array}{cc}
                                   p-1 & 0 \\
                                   0 & -(p-1) \\
                                 \end{array}
                               \right),
$$
thus both points are hyperbolic (two saddles). We are then in the
situation described in Theorems A and C of \cite{LPR96}, and from
the latter theorem we deduce that \eqref{critsyst1} is topologically
equivalent near the origin to \eqref{critsyst2}, which is
furthermore topologically equivalent, through a simple time
rescaling $\tau=(p-1)s$, to the following
\begin{equation}\label{critsyst4}
\begin{split}
&\pm x'=x^2-z^2,\\
&\pm z'=2xz.
\end{split}
\end{equation}
It is easily seen that the
origin is an elliptic point (see \cite{Ar87} for a proof). Coming back to the original system, it
follows that the orbits coming out or entering the point with
$\theta=\pi/2$ describe elliptic cycles.

We study now the second critical point at infinity, with polar angle
$\pi/4$. Since this point is hyperbolic (as we shall see), we pass
to the Poincaré sphere using an equivalent approach used for example
in \cite{Hu}. We let
$$
w=\frac{\varrho}{1-\varrho}\cos\Phi, \quad
k=\frac{\varrho}{1-\varrho}\sin\Phi,
$$
whence
$$
\varrho=\frac{\sqrt{w^2+k^2}}{1+\sqrt{w^2+k^2}}=1-\frac{1}{1+\sqrt{w^2+k^2}},
\quad \Phi=\arctan\frac{k}{w}.
$$
By a direct differentiation, we find that
\begin{equation}\label{interm1}
\varrho'=(1-\varrho)^2(w'\cos\Phi+k'\sin\Phi), \quad
\Phi'=\frac{wk'-kw'}{w^2+k^2}.
\end{equation}
Then, we arrive to the
following equations
\begin{equation}\label{ODEsyst.infty}
\left\{\begin{array}{ll}\varrho'=\frac{\varrho(1-\varrho)\left[(1-\varrho)^2+\varrho^2\right]\left[2(p-1)\varrho^2\cos\Phi\sin\Phi+(2-p)((1-\varrho)^2+\varrho^2\cos^2\Phi)\right]\sin^2\Phi}
{\left[(1-\varrho)^2+\varrho^2\cos^2\Phi\right]\left[(1-\varrho)^2+\varrho^2\cos^2\Phi+(p-1)\varrho^2\sin^2\Phi\right]},\\
\\
\Phi'=\frac{\left[(1-\varrho)^2+\varrho^2\right]\left[2(p-1)\varrho^2\cos\Phi\sin\Phi+(2-p)((1-\varrho)^2+\varrho^2\cos^2\Phi)\right]\sin\Phi\cos\Phi}
{\left[(1-\varrho)^2+\varrho^2\cos^2\Phi\right]\left[(1-\varrho)^2+\varrho^2\cos^2\Phi+(p-1)\varrho^2\sin^2\Phi\right]}-1.\end{array}\right.
\end{equation}
In this setting, our critical point becomes $P_1=(1,\pi/4)$. The
linearization of the system \eqref{ODEsyst.infty} near $P_1$ has the
matrix
$$
M\left(1,\frac{\pi}{4}\right)=\left(
             \begin{array}{cc}
               -1 \ & 0 \\
               0 \ & 2 \\
             \end{array}
           \right),
$$
thus this point is a saddle. There exists a special orbit of
\eqref{ODEsyst1} going into $P_1$ and coming from the interior of
the plane. All other orbits are included in the boundary
$\varrho=1$ and are not of interest to us. The orbit entering $P_1$
and coming from the plane do this in infinite time and its behavior
is
\begin{equation}\label{interm3}
\lim\limits_{t\to\infty}\frac{w'(t)}{w(t)}=1, \quad
\lim\limits_{t\to\infty}w(t)=\infty.
\end{equation}
The solution corresponding to this orbit is increasing for $t>t_0$
large and it satisfies \eqref{interm3}.

\begin{proof}[Proposition \ref{prop.loc0}]
We analyze the special orbit entering $P_1$ in the previous phase
plane. Since $\lim\limits_{t\to\infty}w(t)=\infty$, then, exchanging
if necessary $h$ by $\pi-h$, we may assume that
$\lim\limits_{r\to0}h(r)=0$. To complete the proof, it suffices to
show that there exists $\lim\limits_{r\to0}h'(r)\in\real$.

In the following step, we want to prove that, along this orbit, we
have $w'(t)=k(t)\geq w(t)$ for $t$ sufficiently large. To do this,
assume by contradiction that there exists some $\e>0$ and some
$t_0\in(0,\infty)$ such that $w(t_0)=(1+\e)k(t_0)$. Define
$$l(t):=w(t)-(1+\e)k(t).$$ Thus $l(t_0)=0$. Using
\eqref{ODEsyst1} and the equality $w(t_0)=(1+\e)k(t_0)$, by
straightforward calculations we obtain (everything being taken at
$t=t_0$ that we omit from the notation for simplicity)
\begin{equation}\label{major}
\begin{split}
l'&=k\left[(1+\e)^2+1-(1+\e)\right.\\&\left.\times\frac{\left[(1+k^2+(1+\e)^2k^2)\right]\left[2(p-1)(1+\e)k^2+(2-p)(1+(1+\e)^2k^2)\right]}
{\left[(1+(1+\e)^2k^2)\right]\left[1+(1+\e)^2k^2+(p-1)k^2\right]}\right]\\
&=k\Psi(k,p,\e).
\end{split}
\end{equation}
In order to bound from below $\Psi(k,p,\e)$ above, we first notice
easily that $\frac{\partial}{\partial p}\Psi(k,p,\e)\geq0$, whence
$$\Psi(k,p,\e)\geq\Psi(k,1,\e)=1+\e+\e^2-\frac{(1+\e)k^2}{1+(1+\e)^2k^2}.$$
But the right hand side of the last inequality is decreasing with
$k$, therefore we can take the limit as $k\to\infty$ and get
$$\Psi(k,p,\e)\geq1+\e+\e^2-\frac{1}{1+\e}=\frac{\e(2+2\e+\e^2)}{1+\e}>0,$$
which implies
$$l'(t_0)\geq\frac{k(t_0)\e(2+2\e+\e^2)}{1+\e}>0,$$ hence
$w(t)\geq(1+\e)k(t)$ for any $t\geq t_0$. This is a contradiction
with the fact that $w(t)/k(t)\to 1$ as $t\to\infty$. Thus
$k(t)=w'(t)\geq w(t)$.

We next follow as in \cite{dPGM} by writing
\begin{equation}\label{interm4}
h'(r)=\frac{rh'(r)}{\sin h(r)\cos h(r)}\frac{\tan
h(r)}{r}\cos^2h(r).
\end{equation}
On one hand
\begin{equation}\label{interm5}
\lim\limits_{r\to 0^+}\frac{rh'(r)}{\sin h(r)\cos
h(r)}=\lim\limits_{t\to\infty}-\frac{f'(t)}{\sin f(t)\cos
f(t)}=\lim\limits_{t\to\infty}\frac{w'(t)}{w(t)}=1.
\end{equation}
On the other hand, we have
$$
\frac{d}{dt}(w(t)e^{-t})=e^{-t}(w'(t)-w(t))\geq0,
$$
for any $t$ sufficiently large, by the previous step. This implies
\begin{equation}\label{interm6}
\lim\limits_{r\to 0^+}\frac{\tan
h(r)}{r}=\lim\limits_{t\to\infty}\frac{1}{w(t)e^{-t}}=C\in[0,\infty).
\end{equation}
From \eqref{interm4}, \eqref{interm5}, \eqref{interm6} and the fact
that $h(r)\to0$ as $r\to 0^+$, we deduce that $h'(r)\to C>0$ as
$r\to 0^+$, which shows that $h\in C^1([0,\e))$ for some $\e>0$ small.
The monotonicity in a right-neighborhood of 0 is now obvious from
the previous analysis.
\end{proof}

\subsection{The case $p>2$}\label{subsec.loc02}

In this case, the behavior near the origin is slightly more
complicate. More precisely, we prove
\begin{proposition}\label{prop.loc02}
Let $p>2$ and $h$ be a nonconstant solution to \eqref{ODE} in an
interval $(0,\e)$ for some $\e>0$, and such that $h(r)\in(0,\pi)$
for all $r\in(0,\e)$. Then $h\in C^1([0,\e))$, $h$ is monotone in a
right-neighborhood of $r=0$ and
\begin{equation}\label{lim0.ODE}
\lim\limits_{r\to 0^+}h(r)=0 \quad {\rm or} \quad
\lim\limits_{r\to 0^+}h(r)=\frac{\pi}{2} \quad {\rm or}
\lim\limits_{r\to 0^+}h(r)=\pi.
\end{equation}
\end{proposition}
\begin{proof}
We do the same changes of variable as in Subsection
\ref{subsec.loc0} until arriving to the system \eqref{ODEsyst1}. We
study the phase plane associated to \eqref{ODEsyst1}; the critical
points are the same as for $1<p<2$, that is, a unique critical point
in the plane (the origin) and the six critical points at infinity
corresponding to the polar angles $\theta$ given in \eqref{polar}.

Since now $2-p<0$, the origin becomes in this case a stable node. By
standard theory, this point is asymptotically stable \cite[Section
2.9]{Pe}, that is, there exists a ball $B(0,\delta)$
for some $\delta>0$ such that all the orbits entering $B(0,\delta)$
end up at $(w,k)=(0,0)$. The local analysis of the critical points
at infinity is the same as in Subsection \ref{subsec.loc0}. By
performing the same analysis as in the proof of Proposition
\ref{prop.loc0}, we find the following two different types of orbits:

\noindent $\bullet$ either the orbit entering the saddle point $P_1$
of polar angle $\theta=\pi/4$ and coming from the plane, that in
initial variables means $\lim\limits_{r\to 0^+}h(r)\in\{0,\pi\}$ as in
the proof of Proposition \ref{prop.loc0};

\noindent $\bullet$ or the orbits entering the attractor $O$, meaning
that $w(t)\to0$ as $t\to\infty$. In initial variables, it means that
$\lim\limits_{r\to 0^+}h(r)=\pi/2$, ending the proof.
\end{proof}

\section{Uniqueness of the minimizer}\label{sec.unique}

We prove in this section the uniqueness of the smooth minimizer.
More precisely:
\begin{proposition}\label{prop.unique}
Let $l\in(0,\pi/2]$. Then there exists a unique positive,
non-constant and non-decreasing solution to the problem
\begin{equation}\label{problem1}
\left\{\begin{array}{ll}\begin{split} \displaystyle(p-1)h'^{2}h''
&+(p-3)\left[h'^2\frac{\sin h\cos
h}{r^2}-h'\frac{\sin^{2}h}{r^3}\right]\\&+\frac{h'^3}{r}+h''\frac{\sin^{2}h}{r^2}-\frac{\sin^{3}\cos
h}{r^4}=0, \quad r\in(0,1),\end{split} \\\\
h(1)=l.\end{array}\right.
\end{equation}
Furthermore, $h$ is increasing.
\end{proposition}
\begin{proof}
The proof will be divided into several steps.

\medskip

\noindent \textbf{Step 1. Monotonicity.} We pass again to
logarithmic coordinates as in the previous section by letting
$f(t)=h(e^{-t})$. Then, \eqref{problem1} transforms into
\eqref{ODE2} posed for $t\in(0,\infty)$, with boundary conditions
$$
f(0)=l\in(0,\pi/2], \quad \lim\limits_{t\to\infty}f(t)=0,
$$
with $f$ positive, non-constant and non-increasing. We next show
that, under these conditions, any solution $f$ to \eqref{ODE2} is in
fact decreasing. Assume by contradiction that there exists
$t_0\in(0,\infty)$ such that $f'(t_0)=0$. If $f(t_0)=\pi/2$, by
standard ODE arguments of uniqueness, we have $f\equiv\pi/2$,
contradiction. Thus $0<f(t_0)<\pi/2$. Replacing $t=t_0$ in
\eqref{ODE2}, we get
$$
f''(t_0)=\sin f(t_0)\cos f(t_0)=\frac{1}{2}\sin2f(t_0)>0,
$$
therefore $f$ is increasing in a small right-neighborhood of $t_0$,
contradiction. Thus $f'(t)<0$ for any $t\in(0,\infty)$, and coming
back to initial variables, it follows that $h$ is increasing.

\medskip

\noindent \textbf{Step 2. New change of variable.} For
$f\in(0,\pi/2)$, we perform again the change of variable
$$
w(t):=\cot f(t),
$$
thus arriving to Eq. \eqref{ODE3}, with boundary condition
$w(0)=\cot l$. Since we know from the analysis in the previous
section that $w$ is monotone for $t$ sufficiently large, we can
define its inverse $t(w)$ such that $w(t(w))=w$, and we make the
further change of variable
\begin{equation}\label{ch3}
g(w):=\frac{w}{w'(t(w))}.
\end{equation}
Differentiating with respect to $w$, we get
$$
w''(t(w))=\frac{g(w)-wg'(w)}{g(w)^2t'(w)}=\frac{wg(w)-w^2g'(w)}{g(w)^3}.
$$
We substitute these formulas in Eq. \eqref{ODE3} and obtain
\begin{equation}\label{ODE4bis}
\begin{split}
& g(w)-wg'(w)\\&=\frac{\left[g(w)^2(1+w^2)+w^2\right]\left[2(p-1)w^2+(2-p)g(w)(1+w^2)\right]g(w)}{(1+w^2)\left[g(w)^2(1+w^2)+(p-1)w^2\right]}-g(w)^3.
\end{split}
\end{equation}
After performing some technical operations, we write \eqref{ODE4bis}
in a more suitable form:
\begin{equation}\label{ODE4}
g'(w)=\frac{1}{w}\left[g(w)^3-(2-p)g(w)^2+g(w)-F(w,g(w))\right],
\end{equation}
where
\begin{equation}\label{ODE4aux}
F(w,g):=\frac{w^2g\left[2(p-1)(1+w^2)g^2+(2-p)^2(1+w^2)g+2(p-1)w^2\right]}{(1+w^2)\left[g^2(1+w^2)+(p-1)w^2\right]}.
\end{equation}

\medskip

\noindent \textbf{Step 3. Comparison estimates.} Assume by
contradiction that uniqueness as stated in Proposition
\ref{prop.unique} does not hold and pick two different positive
solutions $g_1$ and $g_2$ of \eqref{ODE4}. Adapting ideas from the
case $p=1$ \cite{dPGM}, our next goal is to estimate the difference
$g_1-g_2$. More precisely:
\begin{equation*}
\begin{split}
(g_1-g_2)'(w)&=\frac{1}{w}\left[g_1(w)^3-g_2(w)^3-(2-p)(g_1(w)^2-g_2(w)^2)\right.\\&\left.+(g_1-g_2)(w)-F(w,g_1(w))+F(w,g_2(w))\right]\\
&=\frac{(g_1-g_2)(w)}{w}\left[\left(g_1(w)-\frac{2-p}{2}\right)^2+\left(g_2(w)-\frac{2-p}{2}\right)^2\right.\\&\left.+(g_1g_2)(w)+1-\frac{(2-p)^2}{2}\right]-\frac{1}{w}\left[F(w,g_1(w))-F(w,g_2(w))\right].
\end{split}
\end{equation*}
Taking into account that
$\lim\limits_{w\to\infty}g_1(w)=\lim\limits_{w\to\infty}g_2(w)=1$,
by the analysis in Section \ref{sec.lb0}, we further obtain
\begin{equation}\label{interm7}
\begin{split}
&\frac{(g_1-g_2)'(w)}{(g_1-g_2)(w)}\\&=\frac{1}{w}\left[\left(g_1(w)-\frac{2-p}{2}\right)^2+\left(g_2(w)-\frac{2-p}{2}\right)^2+(g_1g_2)(w)+1-\frac{(2-p)^2}{2}\right]\\
&-\frac{1}{w}\frac{F(w,g_1(w))-F(w,g_2(w))}{(g_1-g_2)(w)}\geq\frac{C}{w}-\frac{1}{w}\frac{F(w,g_1(w))-F(w,g_2(w))}{g_1(w)-g_2(w)},
\end{split}
\end{equation}
for $w$ sufficiently large, where the constant $C$ satisfies
\begin{equation}\label{interm8a}
C\sim\left(\frac{p}{2}\right)^2+\left(\frac{p}{2}\right)^2+1+1-\frac{(2-p)^2}{2}=2p.
\end{equation}
In order to proceed with the integration of the differential
inequality \eqref{interm7}, we only need to estimate conveniently
$$
\frac{F(w,g_1(w))-F(w,g_2(w))}{g_1(w)-g_2(w)}.
$$
This will be our last step in the proof.

\medskip

\noindent \textbf{Step 4. Lipschitz estimate and end of proof.} In
order to estimate the last quotient, we differentiate $F$ defined in
\eqref{ODE4aux} with respect to its second variable $g$:
\begin{equation*}
\begin{split}
\frac{\partial F(w,g)}{\partial g}=\frac{6(p-1)w^2(1+w^2)g^2+2(2-p)^2w^2(1+w^2)g+2(p-1)w^4g}{(1+w^2)\left[(1+w^2)g^2+(p-1)w^2\right]}\\
-\frac{2g\left[2(p-1)w^2(1+w^2)g^3+(2-p)^2w^2(1+w^2)g^2+2(p-1)w^4g\right]}{\left[(1+w^2)g^2+(p-1)w^2\right]^2}.
\end{split}
\end{equation*}
Letting $w\to\infty$ and recalling that
$\lim\limits_{w\to\infty}g(w)=1$, we can compute
\begin{equation}\label{interm9}
\begin{split}
\lim\limits_{w\to\infty}\frac{\partial F(w,g(w))}{\partial g}=2(p-1).
\end{split}
\end{equation}
Fix $\e\in(0,1)$. Then, by \eqref{interm9} we can find $w_0=w_0(\e)$
large enough such that
\begin{equation*}
\left|\frac{F(g_1)-F(g_2)}{g_1-g_2}(w)\right|\leq 2(p-1)+\e,
\end{equation*}
for any $w\geq w_0(\e)$, and at the same time
\begin{equation*}
\begin{split}
C(w)&:=\left[g_1(w)-\frac{2-p}{2}\right]^2+\left[g_2(w)-\frac{2-p}{2}\right]^2+(g_1g_2)(w)+1-\frac{(2-p)^2}{2}\\&>2p-\e,
\end{split}
\end{equation*}
also for any $w\geq w_0(\e)$, the latter resulting from
\eqref{interm8a}. Integrating now \eqref{interm7} on $(w_0(\e),w)$,
we obtain
\begin{equation}\label{interm10}
\begin{split}
\frac{g_1(w)-g_2(w)}{g_1(w_0)-g_2(w_0)}&\geq\left[\frac{w}{w_0}\right]^C\exp\left[-\int_{w_0}^{w}\frac{1}{w}\frac{F(g_1)-F(g_2)}{g_1-g_2}(w)\,dw\right]\geq\left[\frac{w}{w_0}\right]^{\overline{C}},
\end{split}
\end{equation}
for any $w>w_0(\e)$, where
$$
\overline{C}:=2p-2(p-1)-2\e=2(1-\e).
$$
Since $\e<1$ by its choice, we get $\overline{C}>0$, whence
$$
\lim\limits_{w\to\infty}\frac{g_1(w)-g_2(w)}{g_1(w_0)-g_2(w_0)}=\infty,
$$
which is a contradiction with the fact that $g(w)\to1$ as
$w\to\infty$, for any $g$ solution of \eqref{ODE4}. This completes
the proof.
\end{proof}

\begin{proof}[Theorem \ref{th.1}]
By Propositions \ref{existence} and \ref{prop.loc0} for $1<p<2$,
respectively Propositions \ref{existence} and \ref{prop.loc02} for
$p>2$, and Lemmas \ref{lempro} and \ref{positivity}, there exists a
solution $h$ to problem (P), and furthermore $h\in
C^\infty((0,1])\cap C^1([0,1])$ with $h(0)=0$ and $h(1)=\ell$, $h$
is positive and nondecreasing in $(0,1]$, and $h$ solves
\eqref{ODE}. In order to show that $h$ is unique, let $\tilde h$ be
any solution to problem \eqref{P}. By Lemmas \ref{lempro},
\ref{positivity} and \ref{lem-regularity}, $\tilde h\in
C^\infty((0,1])$, $\tilde h$ is positive and non-decreasing in
$(0,1]$, and $\tilde h$ is non-constant since $\tilde h(0)=0$.
Hence, by Proposition \ref{prop.unique}, $\tilde h\equiv h$. As a
consequence $h$ is increasing, and the proof is complete.
\end{proof}

\section{Smooth critical points}\label{sec.cp}

\subsection{The case $1<p<2$}\label{subsec.cp1}

In this subsection we prove Theorem \ref{th.2}. Proceeding along the
ideas of \cite{dPGM}, we study the behavior of the solutions $h$ to
Eq. \eqref{ODE} for $r$ large. To do this, recalling the previous
changes of variable, we further change the direction of the time
axis in Eq. \eqref{ODE3} and we deal with the following problem
\begin{equation}\label{ODE5}
w''=-\frac{(1+w^2+w'^2)\left[(2-p)(1+w^2)-2(p-1)ww'\right]w'}{(1+w^2)(1+w^2+(p-1)w'^2)}-w,
\end{equation}
under initial condition
\begin{equation}\label{ODE5init}
w(0)=0, \ w'(0)=-\a, \quad \a>0.
\end{equation}
We have the following result.
\begin{proposition}\label{prop.osc}
Let $p\in(1,2)$. There exists $\alpha_0>0$ such that for any $\a\leq\alpha_0$, there exists a unique global
solution to the problem \eqref{ODE5}-\eqref{ODE5init}, and
$\lim\limits_{t\to\infty}w(t)=0$. Furthermore, the critical points
of $w$, that is, the solutions of $w'(t)=0$, consist of a sequence
$(t_n)$ such that $w(t_{2n})<0$ and $(t_{2n})$ are local minima,
$w(t_{2n+1})>0$ and $(t_{2n+1})$ are local maxima. Moreover, we have
$w(t_{2n-1})>w_(t_{2n+1})$ for any positive integer $n$, and there
exists $n_0>0$ such that $|w(t_n)|$ is decreasing for $n\geq n_0$.
In particular, the minimum is attained at $t=t_0$.
\end{proposition}

Before proving Proposition \ref{prop.osc}, we transform \eqref{ODE5}
into an autonomous system whose associated phase plane will be
analyzed in order to classify all the solutions. We divide this
analysis into several steps.

\medskip

\noindent \textbf{Step 1. Local analysis of the phase plane.} We
pass again to the autonomous system associated to \eqref{ODE5}
\begin{equation}\label{ODEsyst2}
\left\{\begin{array}{ll}\displaystyle w'=k,\\ \\ \displaystyle
k'=-\frac{(1+w^2+k^2)\left[(2-p)(1+w^2)-2(p-1)wk\right]k}{(1+w^2)(1+w^2+(p-1)k^2)}-w.\end{array}\right.
\end{equation}
Analyzing the phase plane, we find that \eqref{ODEsyst2} has a
unique critical point in the plane, the origin, and the
linearization near the origin has the matrix
$$
C(0,0)=\left(
         \begin{array}{cc}
           0 \ & 1 \\
           -1 \ & -(2-p) \\
         \end{array}
       \right),
$$
which has two complex eigenvalues with real parts $(p-2)/2<0$. Thus,
the origin is a stable node (attractor). From standard theory (see
\cite[Section 2.9]{Pe}), we obtain that all the orbits of the system
that lie in a neighborhood $B(0,\delta)$ of the origin for some
$\delta>0$, will enter the origin in infinite time, meaning that
$$
\lim\limits_{t\to\infty}w(t)=0, \quad
\lim\limits_{t\to\infty}w'(t)=0.
$$
Indeed, if $w$ is defined on a maximal interval $[0,T)$ with
$T<\infty$, then $w$ can be extended to $t=T$ in the $C^1$ class by
simply letting $w(T)=0$, and then continued afterwards, in
contradiction with the maximality of the interval $[0,T)$, whence
$T=\infty$.

\noindent In order to establish the connections in the plane, we
have to study the critical points of the system \eqref{ODEsyst2} at
infinity. Since the system is not polynomial, we proceed as in the
analysis in Section \ref{sec.lb0} and
we find
that the critical points at infinity are given by the polar angles
\begin{equation}\label{critinf3}
\theta\in\left\{\frac{\pi}{2},\frac{3\pi}{2},\frac{3\pi}{4},\frac{7\pi}{4},\arccos\left(\frac{p-1}{\sqrt{p^2-2p+2}}\right),\arccos\left(\frac{p-1}{\sqrt{p^2-2p+2}}\right)\right\}.
\end{equation}
First of all, the local behavior near the nonhyperbolic critical
points with polar angles $\theta=\pi/2$ and $\theta=3\pi/2$ has
already been studied in Section \ref{sec.lb0}. We deduced that they
are elliptic critical point. We denote by $P_3$ the point with
$\theta=3\pi/2$ and $P_3^{*}$ the one with $\theta=\pi/2$ in the
sequel.

To study the local behavior near the other critical points, that are
hyperbolic, we pass again to the Poincaré sphere by letting
$$
w=\frac{\varrho}{1-\varrho}\cos\Phi, \quad
k=\frac{\varrho}{1-\varrho}\sin\Phi.
$$
After straightforward calculations, we obtain the following system
\begin{equation}\label{ODEsyst.infty2}
\left\{\begin{array}{ll}\varrho'=\frac{\varrho(1-\varrho)\left[(1-\varrho)^2+\varrho^2\right]\left[2(p-1)\varrho^2\cos\Phi\sin\Phi-(2-p)((1-\varrho)^2+\varrho^2\cos^2\Phi)\right]\sin^2\Phi}
{\left[(1-\varrho)^2+\varrho^2\cos^2\Phi\right]\left[(1-\varrho)^2+\varrho^2\cos^2\Phi+(p-1)\varrho^2\sin^2\Phi\right]},\\
\\
\Phi'=\frac{\left[(1-\varrho)^2+\varrho^2\right]\left[2(p-1)\varrho^2\cos\Phi\sin\Phi-(2-p)((1-\varrho)^2+\varrho^2\cos^2\Phi)\right]\sin\Phi\cos\Phi}
{\left[(1-\varrho)^2+\varrho^2\cos^2\Phi\right]\left[(1-\varrho)^2+\varrho^2\cos^2\Phi+(p-1)\varrho^2\sin^2\Phi\right]}-1.\end{array}\right.
\end{equation}
In this setting, our critical points at infinity become
$P_1=(1,7\pi/4)$, $P_2=(1,\arccos((p-1)/\sqrt{p^2-2p+2}))$ and their
symmetrics $P_2^*=(1,\arccos(-(p-1)/\sqrt{p^2-2p+2}))$ and
$P_1^*=(1,3\pi/4)$. Due to the symmetry of the system
\eqref{ODEsyst.infty2}, we will analyze only the points $P_1$ and
$P_2$. The linearization of \eqref{ODEsyst.infty2} near the two
points has the matrices
\begin{equation*}
M\left(1,\frac{7\pi}{4}\right)=\left(
                                 \begin{array}{cc}
                                   1 & 0 \\
                                   \frac{2(p-1)}{p} & -2 \\
                                 \end{array}
                               \right), \end{equation*}
\begin{equation*}M\left(1,\arccos\frac{p-1}{\sqrt{p^2-2p+2}}\right)=\left(
                                                                  \begin{array}{cc}
                                                                    -\frac{1}{p-1} & 0 \\
                                                                    \frac{2}{p} & \frac{p^2-2p+2}{p-1} \\
                                                                  \end{array}
                                                                \right).
\end{equation*}
It follows that both points $P_1$ and $P_2$ are saddle points, of
different orientation: there exists a unique orbit going out of
$P_1$ into the plane, while all the other orbits, that enter $P_1$,
are contained in the curve $\varrho=1$, and there exists a unique
orbit entering $P_2$ and coming from the plane, the rest of orbits
that go out from $P_2$ are contained in the curve $\varrho=1$ and
are not of interest to us.

\medskip

\noindent \textbf{Step 2. No blow-up at the critical point $P_2$.} In
this step we prove two preliminary facts that are needed before
performing the global analysis of the previous phase plane. More
precisely, we claim the following:

\noindent \textbf{(1).} There exists a unique positive, non-constant
and decreasing solution to \eqref{ODE5} corresponding to the orbit
of the system \eqref{ODEsyst.infty2} entering the critical point
$P_2$ and coming from the plane.

\noindent \textbf{(2).} The solution corresponding to the orbit
entering $P_2$ does this in infinite time.

In order to prove assertion \textbf{(1)}, we define
$$
\overline{g}(w):=\frac{w}{w'(t(w))}.
$$
Following similar ideas and calculations to those in the proof of
Proposition \ref{prop.unique}, we deduce that $g$ satisfies the
equation
\begin{equation}\label{ODE6}
g'(w)=\frac{1}{w}\left(g(w)^3+(2-p)g(w)^2+g(w)-\overline{F}(w,g(w))\right),
\end{equation}
where
$$
\overline{F}(w,g)=\frac{w^2g\left[-2(p-1)(1+w^2)g^2+(2-p)^2(1+w^2)g-2(p-1)w^2\right]}{(1+w^2)\left[(1+w^2)g^2+(p-1)w^2\right]}.
$$
Assuming that there are two different solutions $w_1$, $w_2$,
corresponding to solutions $g_1$, $g_2$ of \eqref{ODE6}, we follow
the same steps in the proof of Proposition \ref{prop.unique} and by
straightforward calculations, we obtain
\begin{equation}\label{interm7bis}
\begin{split}
&\frac{(g_1-g_2)'(w)}{(g_1-g_2)(w)}\\ &=\frac{1}{w}\left[\left(g_1(w)+\frac{2-p}{2}\right)^2+\left(g_2(w)+\frac{2-p}{2}\right)^2+(g_1g_2)(w)+1-\frac{(2-p)^2}{2}\right]\\
&-\frac{1}{w}\frac{\overline{F}(w,g_1(w))-\overline{F}(w,g_2(w))}{g_1(w)-g_2(w)}\geq\frac{C}{w}-\frac{1}{w}\frac{\overline{F}(w,g_1(w))-\overline{F}(w,g_2(w))}{g_1(w)-g_2(w)},
\end{split}
\end{equation}
for $w$ sufficiently large. Recalling that we deal with orbits
entering $P_2$, we have $\lim\limits_{w\to\infty}g(w)=p-1$, whence the
constant $C$ satisfies
\begin{equation}\label{interm8}
C\sim\frac{p^2}{2}+(p-1)^2+1-\frac{(2-p)^2}{2}=p^2.
\end{equation}
On the other hand, by differentiating $\overline{F}(w,g(w))$ with
respect to its second variable $g$ and passing to the limit as
$w\to\infty$, we easily get that $\overline{F}(w,g(w))$ admits a
Lipschitz constant for $g$ for $w$ large, which is exactly $2(p-1)$.
We skip the proof since it is similar to that of Step 4 in the proof
of Proposition \ref{prop.unique}. Since $p^2>2(p-1)$ for all $p>1$,
reasoning as there, for any $\e>0$, there exists $w_0=w_0(\e)$
sufficiently large such that
\begin{equation}\label{interm15}
\frac{g_1(w)-g_2(w)}{g_1(w_0)-g_2(w_0)}\geq\left(\frac{w}{w_0}\right)^{\overline{C}},
\end{equation}
where $\overline{C}=p^2-2p+2-2\e>0$ for some $\e>0$ small. Fixing
such $\e>0$ and $w_0$, this implies
$$\lim\limits_{w\to\infty}(g_1(w)-g_2(w))=\infty,$$ in contradiction
with the fact that $\lim\limits_{w\to\infty}g_j(w)=p-1$ for $j=1,2$.

In order to prove assertion \textbf{(2)}, we introduce the following
energy associated naturally to \eqref{ODE5}:
\begin{equation}\label{energy1}
E(w):=w^2+w'^2.
\end{equation}
Thus,
\begin{equation}\label{derivenergy}
\frac{d}{dt}E(w(t))=\frac{2(1+w^2+w'^2)w'^2\left[-(2-p)(1+w^2)+2(p-1)ww'\right]}{(1+w^2)(1+w^2+(p-1)w'^2)}.
\end{equation}
The idea is to bound from above the time derivative of $E(w)$ by
some function of $E(w)$. Recalling the notation $w'=k$ and the fact
that $w/k\to p-1$ for the orbit entering $P_2$, we have
\begin{equation*}
\begin{split}
\frac{d}{dt}E(w(t))&=2(1+E(w(t)))\frac{k^2\left[-(2-p)(1+w^2)+2(p-1)wk\right]}{(1+w^2)(1+w^2+(p-1)k^2)}\\
&=2(1+E(w(t)))H(w(t)).
\end{split}
\end{equation*}
Estimating $H(w)$ as $w/k\to p-1$, we find that
$$
\lim\limits_{w\to\infty}H(w)=\frac{1}{p-1},
$$
hence there exists a sufficiently large constant $C>0$ such that
\begin{equation}\label{ineqdiff}
\frac{d}{dt}E(w(t))\leq C(1+E(w(t))), \quad t>t_0
\end{equation}
for some $t_0>0$. By direct integration in \eqref{ineqdiff}, we
deduce that the solution $w$ is global.

\medskip

\noindent \textbf{Step 3. Global analysis of the phase plane.} In
this step we end the analysis of the phase plane by establishing the
connections between critical points. We restrict ourselves by
symmetry to critical points in the hyperplane $w'\leq0$.

As we already know, there is a unique orbit going out of $P_1$ into
the plane and a unique orbit entering $P_2$ from the plane. We show
that the orbit coming from $P_1$ is attracted by the origin. Assume
that this does not happen, thus this orbit may end up either at the
elliptic point $(1,3\pi/2)$ or at $P_2$. In the latter case, since
the phase plane has no self-intersections at finite points (by
standard uniqueness arguments), by uniqueness of the orbits in $P_1$
and $P_2$, the orbits entering $0$ have to come from the elliptic
point $P_3=(1,3\pi/2)$, that will intersect the connection between
$P_1$ and $P_2$, contradiction. Finally, if the orbit from $P_1$
goes to the elliptic point $P_3$, the plane would remain incomplete,
since the remaining two points are both attracting orbits from the
plane. Hence, it remains the unique possibility that $P_1$ is
connected to the origin, and this is a global solution. Then, the
plane is completed by the fact that the orbits from the elliptic
point $P_3$ go one of them to $P_2$ and the rest either to the
origin or back to the elliptic point $P_3$, as shown in Figure
\ref{figure1}. We will let then, $\alpha_0:=-w'(t_0)$, where $w$
connects $P_3$  and $P_2$ and $t_0$ is such that $w(t_0)=0$ and
$w(t)<0$ for $t>t_0$.

Thus, there are several global orbits. One of them is the one that
connects $P_1$ and $O$ which (before the first intersection with the
axe $w=0$) corresponds to the unique minimizer. Coming back to $h$,
we find that $h(0)=0$, $h(t)\to\pi/2$ as $t\to\infty$. We will study
its oscillatory properties in the sequel. All the other global
orbits come from the elliptic point $P_3$ and reach it in a
countable infinite number of times either forever or  before being
attracted by $P_2$ or the origin. In any case, when coming back to
$h$, all of them verify that $\lim\limits_{r\to 0^+} h(r)=-\infty$.
These connections can be seen better in Figures \ref{figure1} and
\ref{figure2} below.

\medskip

\begin{figure}[ht!]
   \begin{center}
   \includegraphics[width=0.95\textwidth]{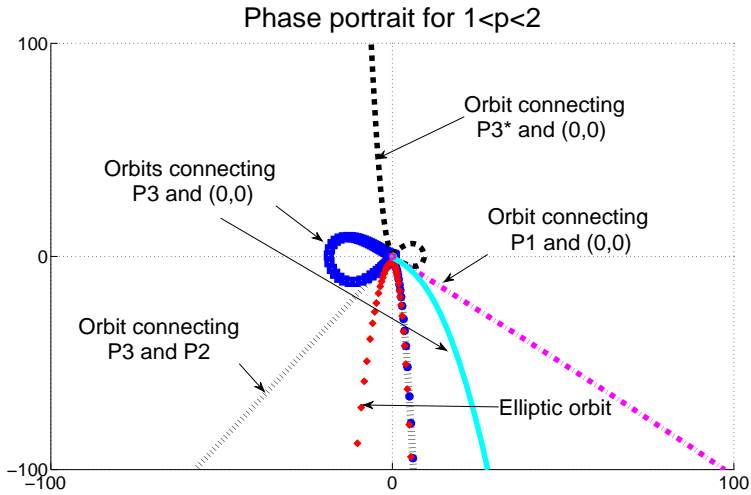}
   \end{center}
   \caption{Global phase portrait of \eqref{ODEsyst2}. Numerical simulation for $p=1+\frac{\sqrt{3}}{3}$. Note that all phase portraits are in fact symmetric with respect to $w'$. However, for clarity reasons, we plot only the orbits joining points in the lower half-plane  plus one orbit in the upper half-plane.}\label{figure1}
\end{figure}

\medskip

\begin{figure}[ht!]
   \begin{center}
   \includegraphics[width=0.95\textwidth]{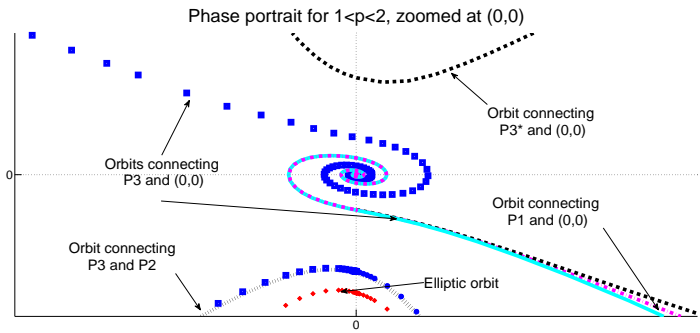}
   \end{center}
   \caption{Phase portrait of \eqref{ODEsyst2} near the origin. Numerical simulation for $p=1+\frac{\sqrt 3}{3}$.}\label{figure2}
\end{figure}

\begin{remark}\label{rem.phase}
  From the previous phase plane analysis, we deduce that \begin{itemize}
    \item[(a)] Let $h^*\in C^\infty((0,1])$ be the solution to \eqref{P} with
$h^*(1)=\frac{\pi}{2}$, as given by Theorem \ref{th.1}.  We pass to
logarithmic coordinates by letting $f(t):=h^*(e^{t})$, $t\in
(0,\infty)$ and $
 w(t):=
 \cot f(t)$. Then, $0>w'(0)>-\alpha_0$.
    \item[(b)] If $w$ is a solution to \eqref{ODE5}-\eqref{ODE5init} with $\alpha=\alpha_0$, then $w$ is a global solution
    and belongs to the orbit entering $P_2$ in the phase plane.
    \item[(c)]  If $w$ is a solution to \eqref{ODE5}-\eqref{ODE5init} with $\alpha>\alpha_0$, then there
        is $t_1$ such that $\lim\limits_{t\to t_1}(w(t),w'(t))= P_3$ and $w(t)<0$, $w'(t)<0$ for $t\in (0,t_1)$.
  \end{itemize}
\end{remark}

\medskip

In order to characterize all global solutions to \eqref{ODE}, we
have, in a preliminary step, to exclude both that the equation
becomes singular or that there is a blow-up in finite time.
\begin{lemma}\label{lem.global}
  Let $h$ be a nonconstant solution of \eqref{ODE} in some interval $(r_0,r_1)\subset(0,+\infty)$. Then, \begin{itemize}
    \item[(i)]$ (p-1)h'(r)^2+\frac{\sin^2h(r)}{r^2}>0$ for all $r\in [r_0,r_1]$.
    \item[(ii)] $|h'(r)|\leq C$ for all $r\in [r_0,r_1]$.
    \item[(iii)] $h$ can be globally extended in $(0,+\infty)$ as a solution to \eqref{ODE}.
  \end{itemize}
\end{lemma}
\begin{proof}
  Obviously, (iii) is implied by standard ODE theory by (i) and (ii). Moreover, both statements are equivalent
  stated in terms of nonconstant solutions in $(t_0,t_1)$ of \begin{equation}\label{ODEexp}
   f''((p-1)f'^2+\sin^2f)+(2-p)f'(f'^2+\sin^2f)=\sin f\cos f((3-p)f'^2+\sin^2f)
 \end{equation} via the smooth transformation $f(t)=h(e^{t})$:
  \begin{itemize}
    \item[$(a)$] $f'(t)^2+\sin^2(f(t))>0$ for all $t\in [t_0,t_1]$.
     \item[$(b)$] $|f'(t)|\leq C$ for all $t\in [t_0,t_1]$.
    \end{itemize}
 Suppose that $f$ is a solution of \eqref{ODEexp} in some finite time interval $(t_0,t_1)$.
 Then, performing the change of variables $w=\cot(f)$, we see that the only points where $(a)$ or $(b)$ may fail are
 $t_0$ in case  $\lim\limits_{t\to
 t_0^{+}}(w(t),w'(t))=(+\infty,-\infty)$ and the points $T\in (t_0,t_1]$ such that $\lim\limits_{t\to
 T^{-}}(w(t),w'(t))=(-\infty,-\infty)$,
 i.\,e. when the orbit of $w$ in the phase portrait approaches the elliptic critical
 point $P_3$.

 \noindent Suppose now, by contradiction, that $\lim\limits_{t\to T^{-}}f'(t)^2+\sin^2 f(t)=0$ (respectively $\lim\limits_{t\to t_0^{+}}f'(t)^2+\sin^2 f(t)=0$).
 Then, since at the elliptic point $P_3$ we have
 $$\lim\limits_{t\to T^-}\frac{\cos f(t)f'(t)}{\sin f(t)}=\lim_{t\to T^-}\frac{-w'(t)}{w(t)}=-\infty,$$ or, respectively,  $$\lim\limits_{t\to t_0^+}\frac{\cos f(t)f'(t)}{\sin f(t)}=\lim_{t\to t_0^+}\frac{-w'(t)}{w(t)}=+\infty,$$
 we obtain that $\sin f(t)=o(f'(t))$ in a left neighborhood of $T$ (respectively in a right neighborhood of $t_0$).
 Then, substituting this in \eqref{ODEexp}, we obtain that $$f''(t)=\frac{p-2}{p-1}f'(t)+o(f'(t)),$$ in a left neighborhood of
 $T$. By a simple integration, we find that $f'(t)\sim K_1e^{-Kt}$ as
 $t\to T^-$ (resp. $t\to t_0^+$), for suitable constants $K>0$, $K_1\in\real$.
 This is a contradiction since $f'(t)\to 0$ as $t\to T^-$ (resp. $t\to t_0^+$). Therefore, $(a)$ is proved.

 For $(b)$, it suffices to note that the right hand
 side of \eqref{ODEexp} is negative when $t\to T^-$
and the second member of the left hand side is positive when $t\to
T^-$. Therefore, $f$ must be concave in a left neighborhood of $T$,
which directly yields $(b)$. In case $\lim\limits_{t\to t_0^+}
|f'(t)|=+\infty$, then, as in case (a), one would obtain that
$$f''(t)=\frac{p-2}{p-1}f'(t)+o(f'(t))\,,$$ in a right neighborhood
of $t_0$. This leads to a contradiction if $t_0$ is finite and thus,
$(b)$ is also proved in this case.
\end{proof}

\begin{proof}[Proposition \ref{prop.osc}] The proof is divided
into two steps in order to ease the reading. We recall that we refer
only to solutions $w$ belonging to orbits that go to the origin
$O(0,0)$ in the phase plane.

\medskip

\noindent \textbf{Step 1. Oscillatory properties.} In order to study
the oscillatory properties of $w$, we need the following preliminary
results

$\bullet$ \textbf{(A)} Given a point $s_*$ such that $w(s_*)=0$ and
$w'(s_*)\neq0$, there exists $t_*\in(s_*,\infty)$ such that
$w'(t_*)=0$ and $w(t_*)<0$ if $w'(s_*)<0$, respectively $w'(t_*)=0$
and $w(t_*)>0$ if $w'(s_*)>0$. Moreover, $t_*$ is a local minimum
(respectively local maximum) point for $w$.

$\bullet$ \textbf{(B)} Given a point $t_*$ such that $w'(t_*)=0$ and
$w(t_*)\neq0$, there exists $s_*\in(t_*,\infty)$ such that
$w(s_*)=0$ and $w'(s_*)>0$ if $w(t_*)<0$, respectively $w(s_*)=0$
and $w'(s_*)<0$ if $w(t_*)>0$.

Assume for the moment that claims (A) and (B) are true, and we
follow with the proof of the Proposition. We apply (A) and (B) in an
iterated form starting from $s_*=s_0=0$, where we know from
\eqref{ODE5init} that $w(0)=0$, $w'(0)=-\a<0$. We thus find an
increasing sequence $(t_n)$ such that $w'(t_n)=0$, $w(t_{2n})<0$,
$w(t_{2n+1})>0$, and, due to (A), $(t_{2n})$ are local minima and
$(t_{2n+1})$ are local maxima. Moreover, $w$ is monotone on
$(t_n,t_{n+1})$ and there exists a point $s_n\in(t_n,t_{n+1})$ such
that $w(s_n)=0$, as (B) implies easily. The fact that $t_n\to\infty$
is easy: if $t_n\to T<\infty$, then by continuity of $w$ and $w'$ we
get $w(T)=w'(T)=0$, thus $w\equiv0$ by general theory, which is a
contradiction.

The fact that the local maxima are ordered, that is,
$w(t_{2n+1})<w(t_{2n-1})$, follows directly from the ordering of the
intersections of the orbit connecting $P_1$ and the origin with the
axis $w'=0$. In particular, the minimum is taken at $t=t_0$ as
stated.

Finally, for the last assertion, we recall the energy $E(w)$
introduced in \eqref{energy1}. The sign of its derivative
\eqref{derivenergy} is determined by the sign of the quantity
$$G(w):=(2-p)(1+w^2)-2(p-1)ww'.$$ From the phase plane analysis, we have that
$\lim\limits_{t\to\infty}G(w(t))=2-p>0$, hence there exists $t_0>0$
sufficiently large such that $G(w(t))>0$ for $t\in(t_0,\infty)$.
Thus, $E(w(t))$ is decreasing on $(t_0,\infty)$. Let $n_0$ be such
that $t_n>t_0$ for $n\geq n_0$. Then $E(t_n)>E(t_{n+1})$, which
implies
$$
|w(t_n)|>|w(t_{n+1})|, \quad \hbox{for} \ \hbox{any} \ n\geq n_0,
$$
ending the proof.

\medskip

\noindent \textbf{Step 2. Proof of statements (A) and (B).} The
proof of claim (A) is immediate from the phase plane. Indeed, if we
have $s_*$ such that $w'(s_*)<0$ and $w(s_*)=0$ and there is no
$t_*\in(s_*,\infty)$ such that $w'(t_*)=0$, it follows that $w'<0$
forever and the orbit of $w$ either goes directly to $O(0,0)$ or it
goes to $P_2$ or $P_3$. The latter is impossible since we deal only
with orbits entering the origin, and the former is impossible since
$w'<0$ implies $w$ decreasing, contradiction. A similar argument is
valid for $s_*$ such that $w'(s_*)>0$ and $w(s_*)=0$.

\medskip

To prove assertion (B), we argue by contradiction. Assume we are at
a point $t_*$ such that $w(t_*)<0$, $w'(t_*)=0$ and there is no
$s_*\in(t_*,\infty)$ as in the statement of (B). Then, either the
orbit of $w$ goes directly to $O(0,0)$ without cutting the axis
$w=0$, or it cuts again the axis $w'=0$. The latter situation is
eliminated in an obvious way, since after cutting again $w'=0$, $w$
will become decreasing, thus it will self-intersect, which is not
allowed in a phase plane. Suppose now that the orbit of $w$ enters
$O(0,0)$ directly. We define
$$
j(t):=\frac{w'(t)}{w(t)},
$$
hence $j(t_*)=0$ and $j'(t_*)=-1$, which is proved by noticing that
$w''(t_*)=-w(t_*)$. We then establish the equation satisfied by $j$,
which is
\begin{equation}\label{ODE7}
j'=-1-j^2-\frac{(1+w^2+w^2j^2)\left[(2-p)(1+w^2)-2(p-1)w^2j\right]j}{(1+w^2)(1+w^2+(p-1)w^2j^2)}.
\end{equation}
Thus, in a sufficiently small neighborhood of the origin we have
$$
j'\leq-1-j^2,
$$
It follows that $j$ blows up to $-\infty$ at some point $s_*>t_*$.
Hence $w(s_*)=0$ and, by the assumption, $w'(s_*)=0$. Coming back to
$f$, this implies that $f(t)=\pi/2$ for $t\geq s_*$. On the other
hand, we know that $f$ is nonconstant, thus, by part (i) in Lemma
\ref{lem.global}, \eqref{ODEexp} is nonsingular and we can apply
standard uniqueness theory to get $f\equiv\pi/2$, contradiction.
This finishes the proof.
\end{proof}

We are now ready to prove Theorem \ref{th.2}.

\begin{proof}[Theorem \ref{th.2}]

(a) and (b) are a direct consequence of Remark \ref{rem.phase} and Proposition \ref{prop.osc}.

\medskip

To prove (c), for $\ell<\frac{\pi}{2}$ we scale $r$ so that $\hat h
(\hat r) = h(\alpha \hat r)$ satisfies $\hat h(1)=\ell$ with
$\alpha>0$. By the scale invariance of \eqref{ODE}, $\hat h$ is also
a solution, and $\hat h$ is increasing and positive in $(0,1)$.
Hence, by Proposition \ref{prop.unique}, it coincides with the
minimizer $h_\ell$ of $ G_p$.

\medskip

To prove (d), we have  to characterize finite global non-constant
solutions $\hat h$ to \eqref{ODE}. Take one of them:  we may assume
without loss of generality that $\hat h(r)\in(0,\pi)$ for $r\in
(r_1,r_2)\subset(0,+\infty)$. Then, we can pass again to
\eqref{ODE5} with $w-$variable in some time interval $(t_1,t_2)$. By
the analysis of the phase plane, since $\hat h$ is finite and
global, then $\hat w(t)=\cot \hat h(e^t)$ is such that $(\hat w(t),
(\hat w)'(t))$ joins $P_1$ and the origin. Since this orbit is
unique, then by the invariance of \eqref{ODE}, $\hat h$ is of the
form in (d).

\medskip

Finally, again by the analysis of the phase plane, there are
infinitely many orbits connecting the elliptic point $P_3$ and the
origin and a unique orbit connecting $P_3$ and $P_2$. When
translating this to the original variable $h$, by Lemma
\ref{lem.global} all these solutions can be globally extended and a
direct computation shows that $h(0)=-\infty$ or $h(0)=+\infty$. In
the former case, again by the invariances of \eqref{ODE}, $h$ is
globally increasing and has either $+\infty$ or $k\pi$ as a limit
for some $k\in \mathbb Z$ (cases $P_3$ and $P_2$) or
$\frac{(2k+1)\pi}{2}$ for some $k\in\mathbb Z$ (in case the orbit
joins $P_3$ and the origin). The latter case is totally symmetric.
\end{proof}

\subsection{The case $p>2$}\label{subsec.cp2}

In this subsection we prove Theorem \ref{th.3}. We start our study
from Eq. \eqref{ODE5} and the associated autonomous system
\eqref{ODEsyst2}. Analyzing the phase plane, we find that the system
has the same critical points as in Subsection \ref{subsec.cp1}: the
origin in the plane and the six different critical points at
infinity, corresponding to polar angles given in \eqref{critinf3},
more precisely, on the Poincaré sphere, $P_1=(1,7\pi/4)$,
$P_2=(1,\arccos((p-1)/\sqrt{p^2-2p+2}))$ and their symmetric points
$P_2^*=(1,\arccos(-(p-1)/\sqrt{p^2-2p+2}))$ and $P_1^*=(1,3\pi/4)$,
together with the elliptic critical points $P_3=(1,3\pi/2)$ and
$P_3^*=(1,\pi/2)$. By symmetry, we analyze only the lower
half-plane.

The linearization of the system \eqref{ODEsyst2} near the origin has
the same matrix $$ C(0,0)=\left(
         \begin{array}{cc}
           0 \ & 1 \\
           -1 \ & -(2-p) \\
         \end{array}
       \right),
$$
which has two complex eigenvalues with real parts $(p-2)/2>0$. Thus,
the origin is now an unstable node. This is the only difference for
the local analysis with respect to the case $1<p<2$, studied in
Subsection \ref{subsec.cp1}. We recall that the critical points at
infinity, whose analysis is identical, are: $P_1$ and $P_2$ saddle
points of opposite orientation, that means, from $P_1$ there is a
unique orbit going out into the plane and there is a unique orbit
entering $P_2$ from the plane. With all these, we are able to
perform the global analysis of the phase plane.

\medskip

\noindent \textbf{Global analysis of the phase plane.} There is a
unique orbit going out from $P_1$ into the plane. This orbit could
enter the elliptic point $P_3$ or $P_2$. Assume by contradiction
that there exists an orbit connecting $P_1$ and $P_2$. Then, the
orbits going out of the unstable node $O(0,0)$ cannot go to the
elliptic point $P_3$, as they would intersect the orbit connecting
$P_1$ and $P_2$, nor to $P_2$, as $P_2$ admits a unique orbit coming
from the interior of the plane, that would be in our assumption the
one coming from $P_1$. But this means that $O(0,0)$ will remain
isolated, contradiction. Thus, the orbit coming from $P_1$ enters
the elliptic point $P_3$. The orbits starting at $O(0,0)$ will thus
enter $P_3$ too, except one of them which will go to $P_2$. We can
thus classify all the orbits of the system \eqref{ODEsyst2} into the
following four types of connections:

\noindent $\bullet$ A unique orbit coming from $P_1$ and connecting
$P_3$. Going back to $h$, it gives rise to an increasing solution
with $h(0+)=0$, and $h(r_0)=\pi$ at some $r_0>0$. As we shall see
later, this orbit will be extended beyond $r_0$ into a global
solution with $h(r)\to\infty$ as $r\to\infty$.

\noindent $\bullet$ An infinite number of orbits starting at
$O(0,0)$ and entering the elliptic point $P_3$. Since, due to
periodicity, we restrict our analysis to $h(0+)\in[0,\pi)$, going
back to $h$, they represent solutions with $h(0)=\pi/2$. These
solutions may oscillate around $\pi/2$ in some finite interval
$(0,r_0)$ and then increase to $h(r_1)=\pi$ for some $r_1>r_0$ (or
decrease to $h(r_1)=0$ for some $r_1>r_0$).

\noindent $\bullet$ A unique orbit coming from $O(0,0)$ and entering
$P_2$. Going back to $h$, this orbit contains solutions with the
following properties: $h(0)=\pi/2$, $h$ oscillates around $\pi/2$ in
some finite interval $(0,r_0)$ for some $r_0>0$ and then it
increases such that $h(r)\to\pi$ as $r\to\infty$ (and by symmetry,
there are also solutions that decrease with $h(r)\to0$ as
$r\to\infty$, the ones corresponding to the orbit connecting
$O(0,0)$ to $P_2^*$, the symmetric of $P_2$ in the plane).

\noindent $\bullet$ Elliptic orbits going forever around the
elliptic point $P_3$. These orbits will contain solutions with
$h(0+)=-\infty$ and $h(r)\to\infty$ as $r\to\infty$.

The connections described above can be seen better in Figures
\ref{figure3} and \ref{figure4} below.

\medskip

\begin{figure}[ht!]
   \begin{center}
   \includegraphics[width=0.95\textwidth]{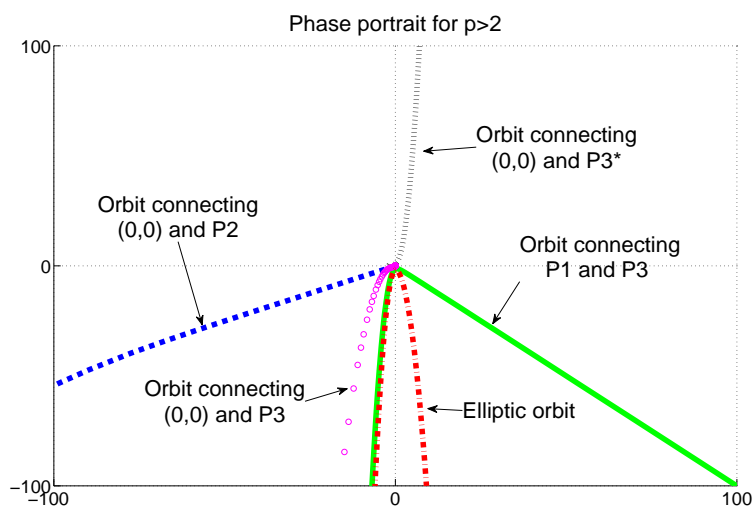}
   \end{center}
   \caption{Global phase portrait of \eqref{ODEsyst2} for $p>2$. Numerical simulation for $p=3$.}\label{figure3}
\end{figure}

\medskip

\begin{figure}[ht!]
   \begin{center}
   \includegraphics[width=0.95\textwidth]{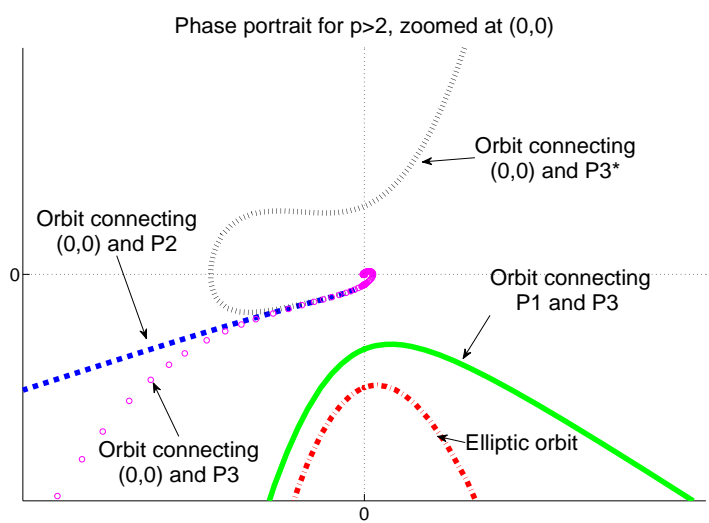}
   \end{center}
   \caption{Phase portrait of \eqref{ODEsyst2} near the origin for $p>2$. Numerical simulation for $p=3$.}\label{figure4}
\end{figure}

\medskip

In order to characterize all global solutions to \eqref{ODE} for
$p>2$, we have to exclude again both that the equation becomes
singular or that there is a blow-up in finite time.

\begin{lemma}\label{lem.global2}
  Let $h$ be a nonconstant solution of \eqref{ODE} in some interval $(r_0,r_1)\subset (0,+\infty)$. Then, \begin{itemize}
    \item[(i)]$ (p-1)h'(r)^2+\frac{\sin^2h(r)}{r^2}>0$ for all $r\in [r_0,r_1]$.
    \item[(ii)] $|h'(r)|\leq C$ for all $r\in [r_0,r_1]$.
    \item[(iii)] $h$ can be globally extended in $(0,+\infty)$ as a solution to \eqref{ODE}.
  \end{itemize}
\end{lemma}
\begin{proof}
 We proceed as in the proof of Lemma \ref{lem.global}.
 First of all, (iii) is implied by (i) and (ii). Next, we pass to logarithmic coordinates by $f(t)=h(e^{t})$
 and to $w=\cot f$, as in Lemma \ref{lem.global}. (ii) is proved as (i) in Lemma \ref{lem.global} with the
 only difference that supposing that $\lim\limits_{t\to T^-} f'(t)=+\infty$ at the elliptic point, we arrive at
 $$f''(t)=\frac{p-2}{p-1}f'(t)+o(f'(t))\,,$$
 which is again a contradiction if $T$ is finite.

 For (i), we note that it is equivalent to prove that $f'$ does not vanish and,
 as in Lemma \ref{lem.global}, it suffices to check it at the elliptic points for the phase portrait of $w$.
 We perform now the following change of variables: $g=\frac{w'}{w^2}$. Then, by \eqref{ODE5}, we have
 $$g'=\frac{w^6g^2((p-2)wg+p-2-(p-1)g^2)}{w(1+w^2)(1+w^2+g^2w^4)}$$$$+\frac{w^4((p-2)g(1+g^2)+(p-5)g^2-1)}{w(1+w^2)(1+w^2+g^2w^4)}$$$$+\frac{2w^2((p-2)wg-1-g^2)+((p-2)wg-1)}{w(1+w^2)(1+w^2+g^2w^4)}.$$
 Noting now that $w(t)\to -\infty$ as $t\to T^-$, $g<0$ in a left neighborhood of $T$, we see that if
 $g(t)\to 0^-$ as $t\to T^-$, then, since $w(t)g(t)=\frac{w'(t)}{w(t)}\to +\infty$ as $t\to T$,
 we get that all terms in the right hand side of the above equation for $g$ are negative.
 Therefore, $g(t)\not \to 0^-$. This yields the conclusion, since $$f'(t)=\frac{-w'(t)}{1+w^2(t)}\sim -g(t)\,,\quad {\rm as \ }t\to T^-.$$
 \end{proof}

Combining Lemma \ref{lem.global2} and the global analysis made
before, we rigorously deduce that the orbits going to the elliptic
point $P_3$ contain indeed global solutions of \eqref{ODE}, as we
already anticipated when doing the analysis.

\medskip

The proof of Theorem \ref{th.3} is analogous to that one of Theorem \ref{th.2} and we skip it.

\medskip

\noindent \textbf{Remark. The special case $p=2$.} The case $p=2$
has been widely studied  with the use of direct methods (see e.g.
the referenced works \cite{BPvdH,BMP}). In this case, there exists a
unique minimizer for the energy and it has (in our notations) the
explicit formula
$$
h(r)=2\arctan\,r.
$$
This particular case can be easily understood with our technique of
phase plane analysis, as it matches perfectly between the pictures
for $1<p<2$ and for $p>2$ respectively. We complete thus the picture
of the evolution of the solutions with respect to the variation of
$p$. We start again from the system \eqref{ODEsyst2} and study the
phase plane. The critical points are the same $O$, $P_1$, $P_2$,
$P_3$, $P_1^*$, $P_2^*$, $P_3^*$ as above, and the local analysis
around the critical points at infinity remains identical. In change,
it can be shown that $O(0,0)$ becomes a center. Thus, passing to the
global analysis of the phase plane associated to the system
\eqref{ODEsyst2}, and restricting our analysis to the lower
half-plane, we notice that:

\noindent $\bullet$ Exactly in this case, the unique orbit going out
from $P_1$ into the plane, enters $P_2$. This orbit contains the
solution $h(r)=2\arctan\,r$.

\noindent $\bullet$ All the rest of orbits are either elliptic
orbits around $P_3$, or orbits that oscillate forever around the
center $O(0,0)$. The latter orbits contain solutions which oscillate
forever around $\pi/2$.

The connections described above can be seen better in Figures
\ref{figure5} and \ref{figure6} below. Indeed, the connection
$P_1-P_2$ is the special and explicit one in this case (and, as we
have seen, it only exists in this case).

\medskip

\begin{figure}[ht!]
   \begin{center}
   \includegraphics[width=0.95\textwidth]{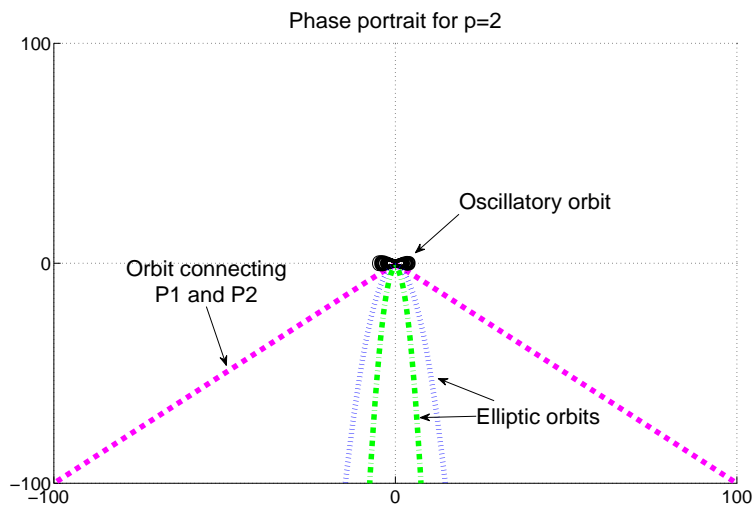}
   \end{center}
   \caption{Global phase portrait of \eqref{ODEsyst2}. Numerical simulation for $p=2$.}\label{figure5}
\end{figure}

\medskip

\begin{figure}[ht!]
   \begin{center}
   \includegraphics[width=0.95\textwidth]{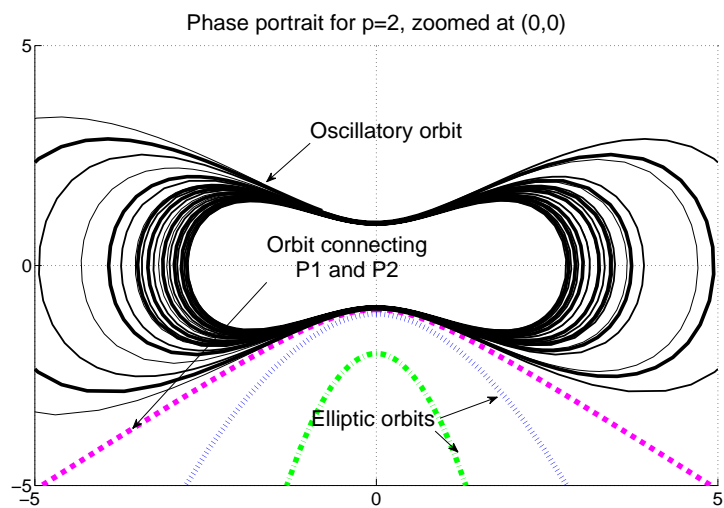}
   \end{center}
   \caption{Phase portrait of \eqref{ODEsyst2} zoomed near the origin for $p=2$.}\label{figure6}
\end{figure}



\begin{thebibliography}{3}
%
%
\bibitem{Ar87}
J. I. Aranda, \emph{Un método para el análisis del origen y el
infinito en sistemas polinomiales planos}, Actas X CEDYA (Spanish),
\textbf{1987}, 20-25.

\bibitem{BPP}
M. Bertsch, R. Dal Passo and A. Pisante, \emph{Point singularities
and non-uniqueness for the heat flow for harmonic maps}, Comm.
Partial Differential Equations, \textbf{28} (2003), 1135-1160.

\bibitem{BPvdH}
M. Bertsch, R. Dal Passo and R. van der Hout, \emph{Nonuniqueness
for the heat flow of harmonic maps on the disk}, Arch. Ration. Mech.
Anal., \textbf{161} (2002), 93-112.

\bibitem{BMP}
M. Bertsch, C. B. Muratov and I. Primi, \emph{Traveling wave
solutions of harmonic heat flow}, Calc. Var. Partial Differential
Equations, \textbf{26} (2006), no. 4, 489-509.

\bibitem{ButGia}
G. Buttazzo, M. Giaquinta and S. Hildebrant, \emph{One-dimensional Variational Problems, An Introduction}, Oxford Lecture Series in Mathematics and its Applications, 15. The Clarendon Press, Oxford University Press, New York, 1998


\bibitem{chdi} K.-C. Chang and W.-Y. Ding, \newblock{\it A
result on the global existence for the heat flows of harmonic maps
from $D^2$ to $S^2$}, Nematics, J.-M. Coron et al. Eds., Kluwer
Academic Publishers (1990), 37--48.

\bibitem{CDY}
K.-C. Chang, W.-Y. Ding and R. Ye, \emph{Finite time blow-up of heat
flow of harmonic maps from surface}, J. Differ. Geom., \textbf{36}
(1992), 507-515.

\bibitem{chh} Y. Chen, M-C. Hong and N. Hungerbuhler,
\emph{Heat flow of p-harmonic
maps with values into spheres,} Math.Z. \textbf{215} (1994), 25--35.


\bibitem{dPGM}
R. Dal Passo, L. Giacomelli and S. Moll, \emph{Rotationally
symmetric 1-harmonic maps from $D^2$ to $S^2$}, Calc. Var. Partial
Differential Equations, \textbf{32} (2008), 533-554.

\bibitem{dSPG}
A. DeSimone and P. Podio-Guidugli, \emph{On the continuum theory of
ferromagnetic solids}, Arch. Rational Mech. Analysis, \textbf{136}
(1996), 201-233.

\bibitem{GM}
L. Giacomelli and S. Moll, \emph{Rotationally symmetric 1-harmonic
flows from $S^2$ to $D^2$: local well-posedness and finite time
blow-up}, SIAM J. Math. Anal., \textbf{42} (2010), no. 6, 2791-2817.

\bibitem{GKY1}
Y. Giga, Y. Kashima and N. Yamazaki, \emph{Local solvability of a
constrained gradient system of total variation}, Abstr. Appl. Anal,
\textbf{8} (2004), 651-682.

\bibitem{GK}
Y. Giga and H. Kuroda, \emph{On breakdown of solutions of a
constrained gradient system of total variation}, Bol. Soc. Parana
Mat., \textbf{22} (2004), no. 3, 9-20.

\bibitem{GKY2}
Y. Giga, H. Kuroda and N. Yamazaki, \emph{Global solvability of
constrained singular diffusion equation associated with essential
variation}, in Free Boundary Problems, Internat. Ser. Numerical
Math. \textbf{154}, Birkhauser, Basel, 2007, 209-218.

\bibitem{vdH}
R. van der Hout, \emph{Flow alignment in nematic liquid crystals in
flows with cylindrical symmetry}, Differential Integral Equations,
\textbf{14} (2001), 189-211.

\bibitem{Hu}
J. Hulshof, \emph{Similarity solutions of the porous medium
equations with sign changes}, J. Math. Anal. Appl., \textbf{157}
(1991), no. 1, 75-111.

\bibitem{Hbook}
N. Hungerbuller, \emph{Heat flow into spheres for a class of
energies}, in Variational Problems in Riemannian Geometry, Progress
Nonlinear Differential Equations Appl. \textbf{59}, Birkhauser,
Basel, 2004, 45-65.

\bibitem{KWC}
R. Kobayashi, J. A. Warren and W. C. Carter, \emph{A continuum model
for grain boundaries}, Phys. D. \textbf{140} (2000), 141-150.

\bibitem{LPR96}
J. Llibre, J. S. Pérez del Rio and J. A. Rodriguez, \emph{Structural
stability of planar homogeneous polynomial vector fields:
applications to critical points and to infinity}, J. Differential
Equations, \textbf{125} (1996), 490-520.

\bibitem{Misawa}
M. Misawa, \emph{On the $p$-harmonic flow into spheres in the
singular case}, Nonlinear Anal., \textbf{50} (2002), 485-494.

\bibitem{Pe}
L. Perko, \emph{Differential Equations and Dynamical Systems. Third
edition}, Texts in Applied Mathematics, \textbf{7}, Springer-Verlag,
New York, 2001.

\bibitem{SapiroBook}
G. Sapiro, \emph{Geometric Partial Differential Equations and Image
Analysis}, Cambridge University Press, Cambridge, UK, 2001.

\end{thebibliography}


\end{document}